\documentclass[12pt]{article}
\usepackage{amsmath}
\usepackage{graphicx}
\usepackage{natbib}
\usepackage{url} % not crucial - just used below for the URL 
\usepackage{xcolor}
\usepackage[colorlinks,linkcolor=blue,citecolor=blue,urlcolor=blue,filecolor=blue,backref=page]{hyperref}
\usepackage{mathrsfs}
\usepackage{subfigure}
\usepackage{multirow}
\usepackage{amsfonts}
\usepackage{multirow}
\usepackage{authblk}
\usepackage{amsmath}
\usepackage{mathrsfs} 
\usepackage{amsfonts}
\usepackage{amsthm}
\usepackage{bbm}
\usepackage{mathtools} 
\usepackage{xcolor}
\usepackage{cleveref}

\allowdisplaybreaks
\usepackage{graphicx}
\usepackage{float}
\usepackage{amsfonts}
\usepackage{mathrsfs}
\usepackage{fancyhdr}
\usepackage{amsthm}
\usepackage{cases}
\usepackage{amsmath,amssymb,bm}
\usepackage{tikz}
\usepackage{tkz-euclide}

\usepackage{enumitem}
\usepackage{xcolor}
\usepackage{upgreek}
\usepackage{empheq}

\def\bb0{\mathbb{0}}

\def\RR{\mathbb{R}}

\def\beq{\begin{eqnarray}}
\def\eeq{\end{eqnarray}}
\def\bsp{\begin{equation}
\begin{split}}
\def\esp{\end{split}
\end{equation}}
\def\ba{\begin{align}}
\def\ea{\end{align}}

\newcommand{\lc}{\left(}
\newcommand{\rc}{\right)}
\newcommand{\lk}{\left[}
\newcommand{\rk}{\right]}

\theoremstyle{Def}

\numberwithin{Thm}{section}

\numberwithin{equation}{section}

\newtheorem{Theorem}{Theorem}
\newtheorem{theorem}[Theorem]{Theorem}
\newtheorem{definition}[Theorem]{Definition}
\newtheorem{lemma}[Theorem]{Lemma}

\newtheorem{corollary}[Theorem]{Corollary}

\newtheorem{proposition}[Theorem]{Proposition}
\newtheorem{example}[Theorem]{Example}
\newtheorem{assumption}[Theorem]{Assumption}

\newtheorem{remark}[Theorem]{Remark}

\newcommand {\bX} {{\bf{X}}}

\newcommand {\tP} {P}
\newcommand {\E} {\mathbb{E}}
\newcommand {\Var} {\text{Var}}
\newcommand {\NGGP} {\text{NGGP}}
\newcommand {\GDP} {\text{GDP}}
\newcommand {\Dir} {\text{Dir}}
\begin{document}

\def\spacingset#1{\renewcommand{\baselinestretch}%
{#1}\small\normalsize} \spacingset{1}

%%%%%%%%%%%%%%%%%%%%%%%%%%%%%%%%%%%%%%%%%%%%%%%%%%%%%%%%%%%%%%%%%%%%%%%%%%%%%%

%\if0\blind
%{
%  \title{\bf Bayesian Nonparametric Mixtures of Exponential Random Graph Models for Ensembles of Networks}
%  \author{Sa Ren\thanks{
%    The first author gratefully acknowledges the Graduate Teaching Assistant scholarship from University of Kent; corresponding author: Sa Ren (sr685@kent.ac.uk)}\\
%    School of Mathematics, Statistics and Actuarial Science, University of Kent, United Kingdom\\
%    Xue Wang \\
%    Walsn Limited, United Kingdom\\
%   Peng Liu\\
%    School of Mathematics, Statistics and Actuarial Science, University of Kent, United Kingdom\\
%    Jian Zhang\\
%    School of Mathematics, Statistics and Actuarial Science, University of Kent, United Kingdom\\
%}
%  \maketitle
%} \fi
%
%\if1\blind
%{
%  \bigskip
%  \bigskip
%  \bigskip
%  \begin{center}
%    {\LARGE\bf Title}
%\end{center}
%  \medskip
%} \fi

%%Note2 to Sarah, use this to Arxiv
\title{\bf Large sample asymptotic analysis for normalized random measures with independent increments}
\author[1]{Junxi Zhang \thanks{junxi3@ualberta.ca}}
\author[2]{Yaozhong Hu \thanks{yaozhong@ualberta.ca}}

\affil[1, 2]{Department of Mathematical and Statistical Sciences,  University of Alberta, Edmonton, Alberta, Canada, T6G 2G1}

\maketitle

\bigskip
\begin{abstract}
Normalized random measures with independent increments represent a large class of Bayesian nonaprametric priors and are widely used in the Bayesian nonparametric framework.  In this paper,  we provide the posterior consistency analysis for  normalized random measures with independent increments (NRMIs) through the corresponding  L{\'e}vy intensities  
%which can be 
used to characterize the completely random measures in the construction of NRMIs.  Assumptions are introduced   on the L{\'e}vy intensities to analyse  the posterior consistency of NRMIs    and are  verified with multiple interesting examples. %Furthermore, we derive 
A focus of the paper is the Bernstein-von Mises theorem for the normalized generalized gamma process (NGGP) when the true distribution of the sample is discrete or continuous. When  the Bernstein-von Mises 
theorem is applied  to construct credible sets,  in addition  to  the usual form  
   there will be an additional  bias term  on the left endpoint  closely related to the number of atoms of  the true distribution when it is discrete.   
%   Thus, there should be a bias correction when construct the confidence intervals by the Bernstein-von Mises results. 
   We also discuss the affect of the estimators for the model parameters of the NGGP under the Bernstein-von Mises convergences.  
Finally, to further explain the necessity  of adding the  bias correction in constructing
credible sets, 
 we illustrate numerically  how the bias correction affects the coverage of the true value by the credible sets when the true distribution is discrete.
\end{abstract}
 
\noindent%
 
\spacingset{1} % DON'T change the spacing!
%Note1 to Sarah change this to 1.5 for JCGS
\section{Introduction}\label{sec:intro}
Bayesian nonparametrics has been undergone major investigation  due to its various applications in many areas, such as biology, economics, machine learning and so on. As a lavish class of Bayesian nonparametric priors, normalized random measures with independent increments (NRMIs),  introduced by \citep{Regazzini_2003}, include the  famous Dirichlet process \citep{ferguson1973}, the $\sigma$-stable NRMIs \citep{kingman1975}, the normalized inverse Gaussian process \citep{lijoi2005b}, the normalized generalized gamma process \citep{lijoi2003normalized, lijoi_2007},   and the generalized Dirichlet process \citep{lijoi2005a}. We refer to \citep{muller2004nonparametric, lijoi2010models, huzhangreview} as reviews of these processes with their properties and applications. 

In Bayesian nonparametric statistics, samples are drawn from a random probability measure that is equipped with a prior distribution. To be more precise, let $(\Omega, \mathcal{F},\mathbb{P})$ be any probability space, let $\mathbb{X}$ be a complete, separable metric space whose $\sigma$-algebra is denoted by $\mathcal{X}$ and let $(\mathbb{M}_{\mathbb{X}}, \mathcal{M}_{\mathbb{X}})$ be the space of all probability measures on $\mathbb{X}$. A sample $\bX=(X_1,\cdots,X_n)$ that takes values in $\mathbb{X}^n$ is drawn  iid from a random probability measure $P$ conditional on $P$, which follows a prior distribution $Q$ on $(\mathbb{M}_{\mathbb{X}}, \mathcal{M}_{\mathbb{X}})$. That is to say,
\begin{align}
 X_1,\cdots, X_n | P \overset{iid}{\sim} P; \quad\quad P \sim Q.  \label{Bayesian nonparametric}
\end{align}
Two natural questions under the literature are raised as follows.
\begin{itemize}
\item[] (i) A frequentist analysis of the Bayesian consistency \citep{freedman1983inconsistent}: by assuming the ``true" distribution of $\bX$ is $P_0$, we are interested in whether the posterior law, that is the conditional law of $P|\bX$, denoted by $Q_n$, converges to $\delta_{P_0}$, the Dirac measure with point mass at the ``true" distribution, as $n \rightarrow \infty$.
\item[] (ii) What is the limiting distribution of centered and rescaled $P|\bX$? In particular, is there a Bernstein-von Mises like theorem and central limit theorem for $P$? If so, what is the limit process of $\sqrt{n}(P|\bX-\E[P|\bX])$?
\end{itemize}
The above two questions are always very important in statistics, as the posterior consistency can guarantee the model behaves ``good" when the sample size is large, and the limiting distribution of the posterior process is the key to construct Bayesian credible sets and conduct hypothesis tests.

%As ones of the most  widely studied and applied  Bayesian nonparametric models, their  distributional properties and computational aspects  have been   extensively studied. However,  the   frequentist properties,    such as posterior consistency, although   previously discussed,    considerations  with respect to their L{\'e}vy intensities have been limited.   Posterior consistency of Bayesian nonparametric models have always been
%the focus of a considerable amount of researches. Most contributions in the literature exploit the “frequentist”
%approach to Bayesian consistency, also termed the “what if”
%method according to \citep{freedman1983inconsistent}. This essentially considers what would happen to the posterior distribution if the sample were generated from a “true” distribution $P_0$: Will the posterior distribution concentrate in suitably
%defined neighborhoods of $P_0$?
Many inspiring works corresponding  to the above questions have been done. Referred to question (i), \citep{james2008} obtains  the posterior consistency analysis  of the two-parameter Poisson-Dirichlet process, which is not an NRMI, but closely  related to NRMIs \citep{pitman1997, perman1992, ghosal2017fundamentals}.  The posterior consistency of the species sampling priors \citep{pitman1996some, aldous1985exchangeability} and the Gibbs-type priors \citep{gnedin2006exchangeable} are discussed in \citep{jang2010} and \citep{Blasi_2013}.  It is worth to point out that there are overlaps among the species sampling priors, the Gibbs-type priors and the homogeneous NRMIs. Whereas, non-homogeneous NRMIs are totally different  from the species sampling priors and the Gibbs-type priors.  As for question (ii), the Bernstein-von Mises results  have been established for the Dirichlet process \citep{lo1983,lo1986remark, ray2021bernstein, huzhang2022functional} and for the two-parameter Poisson-Dirichlet process \citep{james2008, franssen2022bernstein}. Along the same line, we would like to answer the two addressed questions when $P$ is an NRMI. 

Since  NRMIs are  constructed by the normalization of completely random measures \citep{kingman1967completely, kingman1993poisson} associated with  their L{\'e}vy intensities (see e.g.,    section \ref{sec.NRMIs}), it is quite natural to study their properties based on the  corresponding L{\'e}vy intensities.  In this work, we discuss the posterior consistency of non-homogeneous NRMIs (including  the homogeneous case
as a particular case) and provide a simple condition to guarantee the posterior consistency of non-homogeneous NRMIs. As a result, when $P_0$ is continuous, the posterior consistency doesn't hold for NRMIs generally, and when $P_0$ is discrete, the posterior consistency holds as long as our proposed condition is satisfied. 

Furthermore, we  obtain the Bernstein-von Mises theorem for the normalized generalized gamma process (NGGP), which is a flexible class of  Bayesian nonparametric priors includes the Dirichlet process, the normalized inverse-Gaussian process and the $\sigma-$stable process. Through the posterior consistency analysis, the NGGP is posterior consistent when the true distribution $P_0$ is discrete or   when the true distribution $P_0$ is continuous and the parameter $\sigma$ of the  
NGGP  goes to   $0$.   The case that $\sigma \rightarrow 0$ would reduce the NGGP to the Dirichlet process. Thus, we should emphasis the case when the true distribution $P_0$ is discrete.  However, there will be  a bias term on the left hand side of the   Bernstein-von Mises theorem for the NGGP when $P_0$ is discrete. It turns out that the bias term may not go to $0$ when $n \rightarrow \infty$. Thus, in order to construct the ``correct" Bayesian credible sets that cover the true parameter value, we suggest a bias correction to mitigate the bias term. The comparison of credible intervals with bias correction and without bias correction is given in the numerical illustration. In the application, the model parameters of NGGP are chosen by some data driven estimators and  we show that the Bayesian estimator or maximum likelihood estimators 
of the model parameters of the NGGP won't affect the convergences in the Bernstein-von Mises results. 

The outline of this  paper is as follows. In  Section \ref{sec.NRMIs}, we recall  the construction of the NRMIs and their posterior distributions. In Section \ref{sec.posterior consistency},   we  discuss the posterior consistency of the homogeneous NRMIs and  introduce   a simple assumption on  the corresponding L{\'e}vy intensities to guarantee the posterior consistency of the homogeneous NRMIs.  Examples  for several well-known Bayesian nonparametric priors are given to verify  the applicability of the introduced  assumption. In Section \ref{sec. BVM}, we derive the Bernstein-von Mises theorem for the NGGP  and provide an  analysis of the bias correction with an numerical illustration.  Finally, in  Section \ref{sec.discussion}, we provide  a discussion of our results and some ideas that can be studied in the future. In order to ease the flow of the ideas, we delay   the proofs to  the supplementary materials (Section \ref{supplementary}).

\section{Normalized random measures with independent increments}\label{sec.NRMIs}
\subsection{Constructions of NRMIs}
We start by recalling the notions of completely random measures (see e.g.,   \citep{kingman1967completely, kingman1993poisson} and references therein for more details), which play important roles in the construction of NRMIs.
\begin{definition}\label{def: CRM}
Let $\mu$ be a measurable function defined on $(\Omega, \mathcal{F}, \mathbb{P})$ that takes values in $(\mathbb{M}_{\mathbb{X}}, \mathcal{M}_{\mathbb{X}})$. We call $\mu$ is a completely random measure (CRM) if the random variables $\mu(A_1), \cdots, \mu(A_d)$ are mutually independent, for any pairwise disjoint sets $A_1,\cdots,A_d$,  where $d\geq 2$ is a finite integer.
\end{definition} 

The completely random measures play an important role in   Bayesian nonparametric priors  and we   refer to  \citep{Regazzini_2003,lijoi2010models} 
for   more detailed  discussion. 

One way to construct  NRMIs  is through  Poisson random measure 
explained as follows. Denote $\mathbb{S}=\mathbb{R}^+ \times \mathbb{X}$ and denote  its Borel $\sigma$-algebra by $\mathcal{S}$.  A Poisson random measure $\tilde{N}$ on  $\mathbb{S}$   with finite intensity measure $\nu(ds,dx)$ is a random measure  from   $ \Omega\times \mathbb{S} $ to $\mathbb{R}_+$   satisfying 
\begin{enumerate}
\item[(i)] $\tilde{N}(A) \sim \text{Poisson}(\nu(A))$ for any $A$ in $\mathcal{S}$; 
\item[(ii)] for any pairwise disjoint sets $A_1,\cdots,A_m$ in $\mathcal{S}$, the random variables $\tilde{N}(A_1), \cdots, \tilde{N}(A_m)$ are mutually independent.  
\end{enumerate}
The Poisson intensity measure $\nu$  satisfies  the condition (see  \citep{dalay2008book} for details of Poisson random measures) that 
%for any $B \in \mathcal{X}$, 
\begin{align*}
\int_0^{\infty}\int_{\mathbb{X}} \min (s,1)\nu(ds,dx) <\infty\,.
\end{align*}
Let $(\mathbb{B}_{\mathbb{X}},\mathcal{B}_{\mathbb{X}})$ be the space of  finite  measures on $(\mathbb{X},\mathcal{X})$ endowed with the topology of weak convergence and  let  $\tilde{\mu}$ be the random measure defined on $(\Omega, \mathcal{F}, \mathbb{P})$ that takes values in $(\mathbb{B}_{\mathbb{X}},\mathcal{B}_{\mathbb{X}})$ defined as follows,
\begin{align}
\tilde{\mu}(A):=\int_0^{\infty}\int_A s \tilde{N}(ds, dx), \quad \forall A \in \mathcal{X}\,. \label{random measure definition}
\end{align}
It is trivial to verify that $\tilde{\mu}$ is a completely random measure.  It is also well-known that for any $B \in \mathcal{X}$, $\tilde{\mu}(B)$ is discrete and is  uniquely characterized by its Laplace transform as follows: 
\begin{align}
\mathbb{E}\left[e^{-\lambda \tilde{\mu}(B)}\right]=\exp \left\{-\int_0^{\infty}\int_B\left[ 1-e^{-\lambda s}\right]\nu(ds, dx)\right\}\,.   \label{laplace}
\end{align}
The measure $\nu$ is called   the \textit{L{\'e}vy intensity} of $\tilde{\mu}$ and we denote the Laplace exponent 
by 
\begin{equation}
 \psi_{B}(\lambda)=\int_0^{\infty}\int_B\left[ 1-e^{-\lambda s}\right]\nu(ds, dx)\,.  \label{e.psi}
 \end{equation}  
%We refer to \citep{kingman1993poisson} for details and references on completely random measure. 
From the Laplace transform in \eqref{laplace}, we aware that the completely random measure $\tilde{\mu}$ is characterized  completely by its L{\'e}vy intensity $\nu$,    which usually takes the following forms in the literature:  
\begin{itemize}
\item[(a)]   $\nu(ds,dx)= \rho(ds)\alpha(dx)$, where $\rho: \mathcal{B}(\mathbb{R}^+)\to  \mathbb{R}^+$ is some measure on $\mathbb{R}^+$ and $\alpha$ is a non-atomic measure on $(\mathbb{X},\mathcal{X})$ so  that $\alpha(\mathbb{X})=a <\infty$.   The corresponding $\tilde{\mu}$ is called \textit{homogeneous} completely random measure. \label{homogeneous}
\item[(b)]   $\nu(ds,dx)= \rho(ds|x)\alpha(dx)$, where $\rho$ is defined on $ \mathcal{B}(\mathbb{R}^+)\times \mathbb{X}$ such that for any $x\in \mathbb{X}$, $\rho(\cdot|x)$ is a $\sigma$-finite measure on $\mathcal{B}(\mathbb{R}^+)$ and for any $A \in \mathcal{X}$, $\rho(A|\cdot )$ is $\mathcal{B}(\mathbb{R}^+)$ measurable. The corresponding $\tilde{\mu}$ is called \textit{non-homogeneous} completely random measure.\label{non-homogeneous}
\end{itemize}
It is obvious that the case (a) is a special case of case (b).   Usually,  we assume that $\alpha$ is a finite measure so  we   may write  $\alpha(dx)=aH(dx)$ for some   probability measure $H$ and some constant $a=\alpha(\mathbb{X}) \in (0, \infty)$.

To construct NRMIs,   the completely random measure will be normalized, and thus one needs the total mass $\tilde{\mu}(\mathbb{X})$ to be finite and positive almost surely. This happens under  the condition  that $\rho(\mathbb{R}^+)=\infty$ in homogeneous case and  that $\rho(\mathbb{R}^+|x)=\infty$ for all $x\in \mathbb{X}$  in non-homogeneous case \citep{regazzini2002}.   Under  the above conditions, an NRMI $P$ on $(\mathbb{X},\mathcal{X})$ is a random probability measure defined by
\begin{align}
P(\cdot)=\frac{\tilde{\mu}(\cdot)}{\tilde{\mu}(\mathbb{X})}\,. \label{NRMIs}
\end{align}
$P$ is discrete due to the discreteness of $\tilde{\mu}$. For notional simplicity, we let   $T=\tilde{\mu}(\mathbb{X})$  and let $f_{T}(t)$ be the density of $T$ throughout  this paper. 
%We   point  out that $P$ admits a stick-breaking representation (See \citep{sethuraman1994}, \citep{pitman1996some}, \citep{pitman2003poisson}, \citep{favaro2016} for more discussions).
\subsection{Posterior of NRMIs}\label{subsection: posterior of NRMIs}
We will recall the posterior analysis \citep{james2009} of NRMIs, which is a key topic in Bayesian nonparametric analysis.  Let $P$ be  an  
NRMI on  $\mathbb{X}$. A sample of size $n$ from $P$  as in \eqref{Bayesian nonparametric} is an  exchangeable sequence of random variables $\textbf{X}=(X_i)_{i = 1}^n$   defined on $(\Omega,\mathcal{F},\mathbb{P})$  and  taking values in $\mathbb{X}^n$, such that given $P$, $(X_i)_{i \geq 1}$ are drawn iid from distribution $P$, i.e.,
\begin{align}
\mathbb{P}[X_1\in A_1, \cdots, X_n \in A_n|P]=\prod_{i=1}^nP(A_i)\,. \label{sample}
\end{align}
%Then $\textbf{X}=(X_i)_{i = 1}^n$ is a sample of $P$ with sample size $n$. 
Let $\textbf{Y}=(Y_j)_{j=1}^{n(\pi)}$ be the distinct observations of the sample $\textbf{X}$ and  let $n(\pi)$ be the number of unique values of $\textbf{X}$.  This means,  $\pi=(i_1,\cdots, i_{n_1}, \cdots, i_{n_{\pi(n)-1}}, 
\cdots, i_{n_{n(\pi)}}) $ is the partition of $\{1,\cdots,n\}$ of size $n(\pi)$. The number of the $j$th set of the partition is $n_j$, so that $\sum_{j=1}^{n(\pi)} n_j=n$, and $Y_1:=X_{i_1}=\cdots =X_{i_{n_1}}, 
\cdots, Y_{n(\pi)}:=X_{n_{\pi(n)-1}+1}=\cdots=X_{n_{\pi(n) } }$.  
%It it clear that one can represent $\textbf{X}$ by $(\textbf{Y}, \pi)$. 
%Before stating the posterior distribution of NRMIs, we 
Let 
\begin{equation}
 \tau_{k}(u,Y)=\int_0^{\infty}s^{k}e^{-us}\rho(ds|Y)  \quad\hbox{for any positive integer $k$ and $Y \in \mathbb{X}$}.
\label{e.tau} 
\end{equation}   

With these notations, the  posterior distribution of 
 $P$ conditional on the observations of the sample 
$X_1, \cdots, X_n$ is given by the following theorem.
\begin{theorem}[\cite{james2009}] \label{posterior}
Let $P$ be an NRMI with intensity $\nu(ds,dx)=\rho(ds|x)\alpha(dx)$.  The posterior distribution of $P$, given a latent random variable $U_n$, is an NRMI that coincides in distribution with the random measure
\begin{align}
\kappa_n \frac{\tilde{\mu}_{(U_n)}}{T_{(U_n)}} + (1-\kappa_n) \sum_{j=1}^{n(\pi)}\frac{J_j \delta_{Y_j}}{\sum_{j=1}^{n(\pi)}J_j}\,,\label{posterior distribution}
\end{align}
where 
\begin{enumerate}
\item[(i)] The random variable $U_n$ has density
\begin{align}
f_{U_n}(u)=\frac{u^{n-1}}{\Gamma(n)}\int_0^{\infty}t^ne^{-ut}f_T(t)dt\,;   \label{density of U_n}
\end{align}
\item[(ii)] Given $U_n$,  $\tilde{\mu}_{(U_n)}$ is the conditional completely random measure of $
\tilde{\mu}$  with   the  L{\'evy} intensity   $\nu_{(U_n)}=e^{-U_ns}\rho(ds|x)\alpha(dx)$;  
\item[(iii)]   $\{ J_1, \cdots, J_{n(\pi)}\}$  are random variables   depending on $U_n$ and $Y_j$   and  having density
\begin{align}
f_{J_j}(s|U_n=u, \textbf{X}) = \frac{s^{n_j}e^{-us}\rho(s|Y_j)}{\int_0^{\infty}s^{n_j}e^{-us}\rho(ds|Y_j)}\,; 
\end{align}
\item[(iv)] The random elements $\tilde{\mu}_{(U_n)}$ and $J_j$, $j\in \{1, \cdots, n(\pi)\}$ are independent;    
\item[(v)]   $T_{(U_n)}=\tilde{\mu}_{(U_n)}(\mathbb{X})$
and  $\kappa_n=\frac{T_{(U_n)}}{T_{(U_n)}+\sum_{j=1}^{n(\pi)}J_j}$;  
\item[(vi)]   The conditional density of $U_n$ given $\textbf{X}$ is given by
\begin{align}
f_{U_n|\textbf{X}}(u|\textbf{X}) \propto u^{n-1}e^{-\psi(u)} \prod_{j=1}^{n(\pi)}\tau_{n_j}(u,Y_j)\,.\label{posterior density of U_n}
\end{align} 
\end{enumerate}
\end{theorem}
 The above theorem  shows that, given the latent variable $U_n$, the posterior of $P$ is a weighted sum of another NRMI $\frac{\tilde{\mu}_{(U_n)}}{T_{(U_n)}}$ and the normalization of Delta measure  $\delta_{Y_j}$  of distinct observations  $Y_j$,  multiplied by  its corresponding jumps $J_j$.   This gives a rather complete description    of the posterior  distribution  of NRMIs.    More details of the posterior analysis of $\tilde{\mu}$ and $P$ can be found  in \citep{james2009}.
\section{Posterior consistency analysis for the NRMIs}\label{sec.posterior consistency}

In this section, we aim at discussing the posterior consistency for NRMIs as pointed out in question (i) in the introduction. Assume that  $\bX=\{X_1,\cdots, X_n\}$ is a sample from the     ``true'' distribution $P_0$ in $\mathbb{M}_{\mathbb{X}}$.  Namely, $\bX=\{X_1,\cdots, X_n\}$ is  iid  $P_0-$ distributed.   Let $Q_n$ denote the probability law of the posterior random probability measure $P|\bX$.   The   posterior distribution  is said to be weakly  consistent if $Q_n$ concentrates on the weak neighbourhood of $P_0$ almost surely.
%with respect to the infinite product measure $P_0^{\infty}$ as $n$ is large. In other words, 
More precisely, for any weak neighbourhood $O_{\epsilon} \in \mathcal{M}_{\mathbb{X}}$ of $P_0$ with arbitrary radius $\epsilon>0$, 
%we say the posterior of $P$ is weakly consistent if 
\begin{align*}
Q_n(O_{\epsilon}) \rightarrow 1 \qquad a.s.-P_0^{n}\,,
\end{align*}
as $n \rightarrow \infty$. 
The limiting probability measure $P_0^{\infty}=\lim_{n\to \infty}P_0^{n}$ is the infinite product measure on $\mathbb{X}^{\infty}$, namely, $P_0^{\infty}=P_0 \times P_0 \cdots$, which makes the random variables $X_1, X_2, \cdots $ independent with common true distribution $P_0$.

Before presenting the main result, we shall give the following lemma, which provides the moments of the posterior $P$.  The lemma plays an important role in the proof of the main theorem. By recalling $\psi_A $ in \eqref{e.psi}, we denote   
\begin{equation} 
V_{\alpha(A)}^{(k)}(y)= (-1)^k e^{\psi_A(y)} \frac{d^k}{dy^k}e^{-\psi_A(y)} \,,  \label{e.v} 
\end{equation} for any $A \in \mathcal{X}$. 
\begin{lemma}\label{them. moment}
Let $\bX=(X_i)_{i=1}^n$ be a random  sample from a normalized random measure  with independent increments $P$.  The moments and the mixed moments of the
posterior moments of $P$ given $\bX$ are given as follows (we use the notation of Theorem \ref{posterior}). 
\begin{enumerate}
\item[(i)]   For any $A \in \mathcal{X}$ and $m\in \mathbb{N}$, the posterior $m$-th moment of $P$ 
is  given by 
\begin{align}
 \mathbb{E}[(\tP(A))^m|\bX)]=&\frac{\Gamma(n)}{\Gamma(m+n)}\sum_{0 \leq l_1+\cdots +l_{n(\pi)} \leq m}^m {m \choose l_1, \cdots, l_{n(\pi)}} \int_{0}^{\infty}  u^{m}f_{U_n|\bX}(u|\bX) \nonumber\\
& \qquad \qquad V_{\alpha(A)}^{(m-(l_1+\cdots +l_{n(\pi)}))}(u)\left(\prod_{j=1}^{n(\pi)}\frac{\tau_{n_j+l_j}(u,Y_j)}{\tau_{n_j}(u,Y_j)}\delta_{Y_j}(A)\right)du\,.\label{moments}
\end{align}
\item[(ii)]  For any    family of pairwise disjoint subsets  $\{A_1, \cdots, A_q\}$   of  $\mathcal{X}$ and any  integers $\{m_1,\cdots, m_q\}$,   we have
\begin{align}
&\E\left[ \tP(A_1)^{m_1} \cdots \tP(A_q)^{m_q}|\bX \right]=\frac{\Gamma(n)}{\Gamma(m+n)}\int_{0}^{\infty}  u^{m}  f_{U_n|\bX}(u|\bX)\nonumber \\
&\prod_{i=1}^{q+1} \left\{\sum_{0 \leq l_1+\cdots +l_{\#(\lambda_i)} \leq m_i}^{m_i} {m_i \choose l_1, \cdots, l_{\#(\lambda_i)}}V_{\alpha(A_i)}^{(m_i-(l_1+\cdots +l_{\#(\lambda_i)}))}(u)\left(\prod_{j\in \lambda_i}\frac{\tau_{n_j+l_j}(u,Y_j)}{\tau_{n_j}(u,Y_j)}\right)\right\}du\,,
\end{align}
where    $m=\sum_{i=1}^q m_i  $,    $A_{q+1}=(\cup_{i=1}^q A_i)^c$, $m_{q+1}=0$,     $\lambda_i=\{j: Y_j \in A_i\}$ is the set of the index of $Y_j$'s that are in $A_i$,  and 
 $\#(\lambda_i)$ is the number of components in $\lambda_i$. 
\end{enumerate}
\end{lemma}
The above lemma provides the posterior moments of NRMIs. Such results can be reduced to    the moments of NRMIs by letting the sample size $n=0$.  The proof of \cref{them. moment} is inspired by the idea in \citep{james2006conjugacy} and the details are given in the supplementary materials (Section \ref{supplementary}). To apply the above lemma, one needs to deal with the term 
 $V_{\alpha(A)}^{(k)}(y)$   defined by \eqref{e.v}. 
% =\left[(-1)^k \frac{d^k}{dy^k}e^{-\psi_A(y)}\right]e^{\psi_A(y)}$, 
We give the following recursion formula  for  this quantity: 
\begin{align}
V_{\alpha(A)}^{(k)}(y)= \sum_{i=0}^{k-1}{k-1 \choose i}\xi_{k-i}(y)V_{\alpha(A)}^{(i)}(y)\,,\nonumber
\end{align}
where $\xi_{i}(y)= \int_A \tau_{i}(y,x)\alpha(dx)$. 

To answer question (i)   mentioned in the introduction, we    shall study the weak  consistency for more general NRMIs.  To do so, we  need  the following assumption.  
%To simplify  the presentation, we use  $f(u) \overset{u}{\sim} g(u)$ as the notation of $\lim_{u \rightarrow  \infty} \frac{f(u)}{g(u)}=C$ for some finite nonzero  constant $C$ in the remaining part of the  paper.
\begin{assumption}\label{condition A}
 Let   $\tau_k(u,x)$ be  defined by \eqref{e.tau}  and let  $\rho(s|x)$ be a       function such that $u\frac{\tau_{k+1}(u,x)}{\tau_{k}(u,x)} $ is  nondecreasing in $u$ and bounded from above by $k-C_k(x)$ uniformly for all $k\in \mathbb{Z}^+$ and $x\in \mathbb{X}$, where $\{C_k(x)\}$ is a sequence of functions  from $\mathbb{X}$ to  $[0,1)$. Namely, there is an increasing positive function  $\phi(u)$ with $\lim_{u\to \infty} \phi(u)=1$ such that 
 \[
 \sup_{k\in \mathbb{Z}^+\,, x\in \mathbb{X}} \frac{u\frac{\tau_{k+1}(u,x)}{\tau_{k}(u,x)} }{k-C_k(x)}\le \phi(u)\,,\quad \forall \ u\in \RR_+\,.
 \] 
%  for any $x \in \mathbb{X}$. 
 % and $\{h_k(x)\}$ is a sequence of functions such that $0 \leq h_k(x)<\infty$ for any $x \in \mathbb{X}$ and $k\in \mathbb{Z}^+$. \label{condition A}
\end{assumption}
\begin{theorem}\label{consistency thm}
Let $P$ be an NRMI with L{\'e}vy intensity   $\nu(ds,dx)=\rho(s|x)ds\alpha(dx)$, where  $\rho(s|x)$ satisfies   Assumption \ref{condition A}. Then
\begin{itemize}
\item[1.] If $P_0$ is continuous, then the posterior of $P$ converges weakly to a point mass at $\bar{C_1}H(\cdot)+(1-\bar{C_1})P_0(\cdot)$  a.s.$-P_0^{\infty}$,  where $\bar{C_1}$ is the population mean of $\{C_1(X_i)\}_{i=1}^{\infty}$,  that is to say $\bar{C_1}=\lim_{n\rightarrow \infty}\frac{\sum_{i=1}^n C_1(X_i)}{n}$. 
\item[2.] If $P_0$ is discrete with $\lim_{n\rightarrow \infty}\frac{n(\pi)}{n}= 0$, then $P$ is weakly consistent, i.e.,
the posterior of $P$ converges weakly to a point mass at $P_0(\cdot)$ a.s.$-P_0^{\infty}$.
\end{itemize}
\end{theorem}
Although the \cref{condition A} looks complicate, it is quite easy to check as long as $\rho(s|x)$ is given. For instance, the intensities $\rho(s|x)$ for almost all popular NRMIs are gamma type, and we shall check \cref{condition A} for these NRMIs in \cref{example: NGGP}, \cref{example: GDP} and \cref{example: extended gamma} to show how the \cref{condition A} works for these processes.  This   allows more applicability of Theorem \ref{consistency thm}.  

As a comparison between Theorem \ref{consistency thm} and the results in \citep{jang2010} for the species sampling priors and \citep{Blasi_2013} for the Gibbs-type priors, Theorem \ref{consistency thm} considers the consistency results for the non-homogeneous NRMIs, which is a more general class of Bayesian nonparametric priors than both the species sampling priors and the Gibbs-type priors. On the other hand, the conditions in \citep{jang2010, Blasi_2013} are not trivial to verify for homogeneous NRMIs, even though the predictive distribution of homogeneous NRMIs is given \citep{pitman2003poisson, james2009}.

In Theorem \ref{consistency thm}, we require $\lim_{n \rightarrow\infty}\frac{n(\pi)}{n} = 0$ as a condition to guarantee the posterior consistency result when $P_0$ is discrete. This condition is true almost surely by the following proposition.

\begin{proposition}\label{remark: ratio}
When $P_0$ is discrete,  $\lim_{n \rightarrow \infty}\frac{n(\pi)}{n} = 0$,almost surely. When $P_0$ is continuous,  $\lim_{n \rightarrow\infty}\frac{n(\pi)}{n} = 1$,almost surely.
% the posterior of $P$ converges weakly to a point mass at $\bar{C_1}\alpha(\cdot)+(1-\bar{C_1})P_0(\cdot)$ a.s.$-P_0^{\infty}$. However, this is not true almost surely.
\end{proposition}
\begin{proof}
Note that $P_0$ is the true distribution of $\bX$, i.e., $\bX \overset{iid}{\sim} P_0$. Recall that $n(\pi)$ is the number of distinct observations of $\bX$. Let $\mathbb{P}_n(\cdot)=\frac{\sum_{i=1}^n \delta_{X_i}(\cdot)}{n}$ be the empirical probability measure.

 If $P_0$ is discrete, we denote the collection of atoms of $P_0$ is $\mathbb{D}$, then $\mathbb{D}=\{x_1,x_2,\cdots\}$. For any $k\in \mathbb{Z}^+$, we have $n(\pi) \leq k+ n \mathbb{P}_n(\{x_{k+1}, x_{k+2}, \cdots\})$. Thus,
 \begin{align*}
 \lim_{n \rightarrow \infty} \frac{n(\pi)}{n}&=\lim_{k \rightarrow \infty} \lim_{n \rightarrow \infty}\frac{n(\pi)}{n}\\
 &\quad \leq \lim_{k \rightarrow \infty} \lim_{n \rightarrow \infty} \frac{k}{n}+ \mathbb{P}_n(\{x_{k+1}, x_{k+2}, \cdots\})\\
 &\quad =\lim_{k \rightarrow \infty} P_0(\{x_{k+1}, x_{k+2}, \cdots\})=0\,,
 \end{align*}
almost surely, where we use the Borel–Cantelli lemma when taking the limit of $\mathbb{P}_n(\{x_{k+1},\\ x_{k+2}, \cdots\})$ as $n \rightarrow \infty$.

If $P_0$ is continuous, we have $n(\pi)=n\mathbb{P}_n(\mathbb{X})$ and thus $\lim_{n \rightarrow \infty} \frac{n(\pi)}{n}=\lim_{n \rightarrow \infty} \mathbb{P}_n(\mathbb{X})=P_0(\mathbb{X})=1$, almost surely.
\end{proof}

By the identity that $\frac{d}{du}\tau_{k}(u,x)=\frac{d}{du} \int_0^{\infty} s^k e^{-us} \rho_x(s)ds=-\tau_{k+1}(u,x)$, we can have the following assumption that is equivalent to \cref{condition A}.
\begin{assumption}\label{condition B}
 $\rho_x(s)$   is a  function such that $u\frac{d}{du}\ln \left(\tau_{k}(u,x) \right)$ is  nonincreasing in $u$ and bounded from below by $C_k(x)-k$ for all $k\in \mathbb{Z}^+$ and $x\in \mathbb{X}$.
\end{assumption}
\begin{remark}
Theorem \ref{consistency thm} can be extended to more general  NRMIs. For example, \citep{james2002poisson} introduced  the h-biased random measures $\tilde{\mu}$ by $\int_{\mathbb{Y}\times \mathbb{X}}g(s)\tilde{N}(ds,dx)$, where $g:\mathbb{Y} \to \mathbb{R}^+$ is an integrable function on    any complete and separable metric space
$\mathbb{Y}$.
\end{remark}
One interesting quantity to be considered is $n(\pi)$,   the number of distinct observations of the sample $\{X_i\}_{i=1}^n$. In   Bayesian nonparametric mixture models, $n(\pi)$ is   the number of clusters in the sample observations and thus is studied in a number of  works that are   concerning the clustering and so on. Among the literatures let us mention that     the distribution of $n(\pi)$ is  obtained in \citep{Korwar_1973} for the Dirichlet process;  in  \citep{antoniak1974mixtures} 
for the mixture of Dirichlet process; in  \citep{pitman2003poisson} for the two-parameter Poisson-Dirichlet process. For the general NRMIs we have by   a result of \citep{james2009}: 
\begin{proposition}
For any positive integer  $n$, the distribution of $n(\pi)$ is
\begin{align}
\mathbb{P}(n(\pi)=k)=\int_0^{\infty}\frac{nu^n-1}{k!} e^{-\int_{\mathbb{X}}\int_0^{\infty}(1-e^{-ys})\rho(ds|x)\alpha(dx)} \sum_{(n_1,\cdots,n_k)}\prod_{j=1}^k \frac{\int_{\mathbb{X}}\tau_{n_j}(u,x)\alpha(dx)}{n_j!}du,
\end{align}
where $k=1, \cdots, n$,  and  the summation  is over all vectors of positive integers $(n_1, \cdots, n_k)$ such that $\sum_{j=1}^k n_j=n$.
\end{proposition}
As we mentioned above, the  \cref{condition A} is in fact  quite easy to verify. We provide in the following    examples to see the applicability  of Theorem 
\ref{consistency thm}.
\begin{example}\label{example: NGGP}
The normalized generalized gamma process $\NGGP(a, \sigma, \theta, H)$  \citep{lijoi2003normalized, lijoi_2007} is an NRMI with the following homogeneous L\'evy intensity 
\begin{align}
\nu(ds,dx)=\frac{1}{\Gamma(1-\sigma)}s^{-1-\sigma}e^{-\theta s}ds \alpha(dx)\,,\label{NGGP intensity}
\end{align}
where the parameters $\sigma \in (0,1)$ and $\theta>0$. It is easy to see that the Laplace transform for $\tilde{\mu}(A)$ is 
\begin{align*}
\mathbb{E}\left[e^{-\lambda \tilde{\mu}(A)}\right]=\exp \left\{-\frac{\alpha(A)}{\sigma}[(\lambda+\theta)^{\sigma}-\theta^{\sigma}]\right\}\,.
\end{align*}
When  $\theta \rightarrow 0$, this  NRMI yields the homogeneous $\sigma$-stable NRMI introduced by \citep{kingman1975}.  Letting $\sigma \rightarrow 0$, this  NRMI becomes the Dirichlet process \citep{ferguson1973}.    If we let  $\sigma= \theta=\frac{1}{2}$, this  NRMI becomes the normalized inverse-Gaussian process \citep{lijoi2005b}.

It is easy to 
check  that for any nonnegative integer $k$, 
\begin{align*}
\tau_k(u,x)=\tau_k(u)=\frac{1}{\Gamma(1-\sigma)}\int_0^{\infty} s^{k-\sigma-1}e^{-(u+\theta)s}ds=\frac{\Gamma(k-\sigma)}{\Gamma(1-\sigma)(u+\theta)^{k-\sigma}}\,.
\end{align*}
It is obvious  that  $u\frac{\tau_{k+1}(u,x)}{\tau_k(u,x)}=u\frac{k-\sigma}{u+\theta}$ is increasing in $u$ with the upper bound  $k-\sigma$. Thus, the assumption  \ref{condition A}  is verified and  Theorem \ref{consistency thm}  implies the normalized generalized gamma process is posterior consistent   when $\sigma \rightarrow 0$ (i.e. the Dirichlet process), or when $P_0$ is discrete.
\end{example}
\begin{example}\label{example: GDP}
The generalized Dirichlet process $\GDP (a, \gamma, H)$ \citep{lijoi2005a} is an NRMI with the following homogeneous L\'evy intensity 
\begin{align}
\nu(ds,dx)=\sum_{j=1}^{\gamma}\frac{e^{-js}}{s}ds \alpha(dx)\,,\label{GDP intensity}
\end{align}
where $\gamma$ is a positive integer. The corresponding Laplace transform of $\tilde{\mu}(A)$ is 
\begin{align*}
\mathbb{E}\left[e^{-\lambda \tilde{\mu}(A)}\right]=\left(\frac{(\gamma!)}{(\lambda+1)_{\gamma}}\right)^{\alpha(A)}\,,
\end{align*}
where for $c>0$, $(c)_k=\frac{\Gamma(c+k)}{\Gamma(c)}$ is the ascending factorial of $c$ for any positive integer $k$. When  $\gamma=1$, the  generalized Dirichlet process is reduced  to the Dirichlet process.

It is trivial to obtain  for any nonnegative integer $k$, 
\begin{align*}
\tau_k(u,x)=\tau_k(u)=\sum_{j=1}^{\gamma} \frac{k}{(u+j)^k}\,.
\end{align*}
It follows  $\frac{\tau_{k+1}(u,x)}{\tau_k(u,x)}=k\frac{\sum_{j=1}^{\gamma}(u+j)^{-k-1}}{\sum_{j=1}^{\gamma}(u+j)^{-k}} \in (\frac{k}{u+\gamma}, \frac{k}{u+1})$, which implies $u\frac{\tau_{k+1}(u,x)}{\tau_k(u,x)} =u \frac{k}{u+c(\gamma)}$ with some constant $c(\gamma) \in (1, \gamma)$. Thus, $u\frac{\tau_{k+1}(u,x)}{\tau_k(u,x)}$ is increasing in $u$ with the upper bound   $k$. Theorem 
\ref{consistency thm} can then be used to conclude that  the generalized Dirichlet process is posterior consistent.
\end{example}
\begin{example}\label{example: extended gamma}
As a non-homogeneous example, we consider the extended gamma NRMI whose   non-homogeneous L\'evy intensity
is given by 
\begin{align}
\nu(ds,dx)=\frac{e^{-\beta(x)s}}{s}ds \alpha(dx)\,,\label{extended gamma intensity}
\end{align}
where $\beta(x): \mathbb{X} \rightarrow \mathbb{R}^+$
is an integrable  function
(with respect to $\alpha(dx)$). Such NRMI is constructed by the normalization of the  extended gamma process on $\mathbb{R}$ introduced by \citep{dykstra1981bayesian}. More generally, \citep{lo1982bayesian} studied  the extended Gamma process, called  weighted Gamma process
on abstract spaces.

By a trivial computation, for any nonnegative integer $k$, $\tau_k(u,x)=\frac{\Gamma(k)}{(u+\beta(x))^k}$ and thus $u\frac{\tau_{k+1}(u,x)}{\tau_k(u,x)}=u\frac{k}{u+\beta(x)}$ and  the assumption \ref{condition A} is satisfied.  Theorem \ref{consistency thm} implies that   the  extended gamma NRMI is posterior consistent when $\beta(x)$ is integrable with respect to $\alpha(dx)$.

\end{example}
Our theorem can also be applied to more general NRMIs which haven't been investigated in previous works. For example,  we may naturally  consider the following {\it generalized extended gamma NRMI} by letting the L\'evy intensity be 
as follows:  
$$\nu(ds,dx)=\sum_{i=1}^r \frac{e^{-\beta_i(x)s}}{s}ds \alpha(dx),$$ where $r\in \mathbb{Z}^+$ and $\beta_i(x): \mathbb{X} \rightarrow \mathbb{R}^+$ are integrable  functions (with respect to $\alpha(dx)$).   A similar argument to that of \cref{example: GDP} and \cref{example: extended gamma}   implies  that the generalized extended gamma NRMI is posterior consistent when   $\beta_i(x)$ is integrable  (with respect to $\alpha(dx)$) for all $i\in \{1, \cdots, r\}$.  

Relying on the results in this section, we have answered the question (i) addressed in the introduction. The posterior consistency of NRMIs when $P_0$ is continuous doesn't hold generally, as the posterior distribution of  NRMIs is inconsistent when $\bar{C}_1 \neq 0$ or $H \neq P_0 (\mathbb{P}_n)$. However, it is rare to choose $H$ to be the ``true" distribution $P_0$ and it is not possible to let $H=\mathbb{P}_n$ before a sample is observed. Thus, the assumption $\bar{C}_1=0$ should be made to guarantee the posterior consistency for the NRMIs when $P_0$ is continuous. And, whenever $\rho_x(ds)$ is gamma type, $\bar{C}_1=0$ would reduce the corresponding $P$ to the Dirichlet process or the generalized Dirichlet process.
\section{Bernstein-von Mises theorem for the generalized normalized gamma process}\label{sec. BVM}
The Bernstein-von Mises theorem links Bayesian inference with frequentist inference. Similarly to  the Bernstein-von Mises theorem \citep{vaart_1998} in Bayesian parametric framework, one can derive the Bernstein-von Mises theorem in Bayesian nonparametric framework. There has been some works in the literature.  
One example is the Bernstein-von Mises theorem for the empirical process $\mathbb{P}_n=\frac{\sum_{i=1}^n \delta_{X_i}}{n}$ \citep{van1996, vaart_1998}. With the fact that the maximum likelihood estimator of $P_0$ in the Bayesian nonparametric sense is $\mathbb{P}_n=\frac{\sum_{i=1}^n \delta_{X_i}}{n}$, one can conclude the limit law of $\sqrt{n}(\mathbb{P}_n-P_0)$ is normal distribution.  Based on a similar idea, we would consider the limit law of the posterior distribution of $\sqrt{n}(P-\mathbb{P}_n)$ given an iid sample $\bX$ from $P_0$.  To explain the Bernstein-von Mises theorem in the Bayesian nonparametric case, we temporarily let $P \in \mathbb{M}_{\mathbb{X}}$ be any random probability measure and define the functional as follows:
\begin{align*}
Pf=\int f dP, \qquad P_0f= \int f dP_0, \qquad \mathbb{P}_nf= \int f d\mathbb{P}_n=\frac{\sum_{i=1}^n f(X_i)}{n}\,,
\end{align*}
where $f: \mathbb{X} \rightarrow \mathbb{R}$ is any measurable functions. 

Let $\mathbb{F}$ be the collection of functions $f$, the Bernstein-von Mises theorem in the Bayesian nonparametric case considers the distribution of $\{\sqrt{n}(Pf-\mathbb{P}_nf)| \bX: f \in \mathbb{F}\}$ and $\{\sqrt{n}(\mathbb{P}_nf-P_0f) : f \in \mathbb{F}\}$. It is worth to point out that there have been many works  for the  weak convergence of stochastic processes indexed by elements of    Banach space  of functions, we  refer the statisticians  to
 \citep{van1996, vaart_1998} for further reading.   When the function collection $\mathbb{F}$ is finite, both $\{\sqrt{n}(Pf-\mathbb{P}_nf)| \bX: f \in \mathbb{F}\}$ and $\{\sqrt{n}(\mathbb{P}_nf-P_0f) : f \in \mathbb{F}\}$ are random vectors in Euclidean space. Otherwise, it is convenient to consider the $\mathbb{F}$ to be $P_0-$Donsker.   Here we recall that $\mathbb{F}$ is $P_0-$Donsker if the sequence $\sqrt{n}(\mathbb{P}_nf-P_0f)$ converges to $\mathbb{B}_{P_0}^o$ in distribution in the metric space $l^{\infty}(\mathbb{F})$ of bounded functions $g: \mathbb{F} \rightarrow \mathbb{R}$, equipped with the uniform norm $||g||_{\mathbb{F}}= \sup_{f \in \mathbb{F}} |g(f)|$. And $\mathbb{B}_{P_0}^o$ is a Brownian bridge with parameter $P_0$ or $P_0-$ Brownian bridge, so that $\E[\mathbb{B}_{P_0}^of]=0$ and $\E[\mathbb{B}_{P_0}^of_1\mathbb{B}_{P_0}^of_2]=P_0(f_1f_2)-P_0f_1P_0f_2$. An notable result is that a finite set $\mathbb{F}$ is $P_0-$Donsker if and only if $P_0f^2<\infty$ for every $f \in \mathbb{F}$. For the infinite $P_0-$Donsker classes, one can find details and examples in \citep{van1996}.

In order to define the weak convergence of $\sqrt{n}(P-\mathbb{P}_n)$ conditional on $\bX$ to $\mathbb{B}_{P_0}^o$, we  can  use the conditional weak convergence in the bounded Lipschitz metric \citep{van1996} as follows: 
\begin{align}
\sup_{h \in \text{BL}_1}\left| \E \lk h(\sqrt{n}(P-\mathbb{P}_n)|\bX \rk-\E[h(\mathbb{B}_{P_0}^o)]\right| \rightarrow 0\,,\label{eq1. BL metric}
\end{align}
 as $n \rightarrow \infty$. The expectation in \eqref{eq1. BL metric} is taken for the random probability measure $P$, and thus the left side of \eqref{eq1. BL metric} is a function of $\bX$. The convergence in \eqref{eq1. BL metric} refers to the iid sample $\bX$ from $P_0$ and can be in probability or almost surely. The supreme is taken over the set BL$_1$ of all functions $h : l^{\infty}(\mathbb{F}) \rightarrow [0,1]$ such that $|h(f_1)-h(f_2)|\leq ||f_1-f_2||_{\mathbb{F}}$, for all $f_1,f_2 \in l^{\infty}(\mathbb{F})$. We denote the above convergence as 
 \begin{align*}
 \sqrt{n}(P-\mathbb{P}_n)|\bX  \leadsto \mathbb{B}_{P_0}^o\,.
 \end{align*}
 
Under the convergence criteria we explained above, we will present the Bernstein-von Mises theorem  when $P \sim \NGGP (a, \sigma, \theta, H)$. For simplicity of interpretation, let $\tilde{\mathbb{P}}_n=\frac{\sum_{i=1}^{n(\pi)} \delta_{Y_i}}{n(\pi)}$.

\begin{theorem}\label{thm: BVM}
Let $\bX$ be a sample as defined in \cref{Bayesian nonparametric} with $P \sim \NGGP(a, \sigma, \theta, H)$. Let $\mathbb{F}$ be the finite collection of functions such that $P_0f^2 <\infty$ and $Hf^2< \infty$ for any $f \in \mathbb{F}$. We have the following convergences almost surely under $P_0^{\infty}$.
\begin{itemize}
\item[] (i) If $P_0$ is discrete, 
\begin{align}
& \sqrt{n}\left(P- \lk \mathbb{P}_n+\frac{\sigma n(\pi)}{n}(H-\tilde{\mathbb{P}}_n) \rk \right)|\bX  \leadsto \mathbb{B}_{P_0}^o\,,\label{eq1: BVM dicrete}\\
& \sqrt{n}\left(P-\mathbb{E}[P|\bX]\right)|\bX  \leadsto \mathbb{B}_{P_0}^o\,.\label{eq2: BVM dicrete}
\end{align}
\item[] (ii) If $P_0$ is continuous,
\begin{align}
& \sqrt{n}\left(P- \lk (1-\sigma)\mathbb{P}_n+\sigma H \rk \right)|\bX \nonumber\\
&\qquad\qquad\qquad \leadsto \sqrt{1-\sigma}\mathbb{B}_{P_0}^o+\sqrt{\sigma(1-\sigma)}\mathbb{B}_{H}^o+\sqrt{\sigma}Z(P_0-H)\,,\label{eq1: BVM continuous}\\
& \sqrt{n}(P-\mathbb{E}[P|\bX])|\bX  \nonumber\\
&\qquad\qquad\qquad \leadsto \sqrt{1-\sigma}\mathbb{B}_{P_0}^o+\sqrt{\sigma(1-\sigma)}\mathbb{B}_{H}^o+\sqrt{\sigma}Z(P_0-H)\,.\label{eq2: BVM continuous}
\end{align}
\end{itemize}
Here $\mathbb{B}_{P_0}^o$, $\mathbb{B}_{H}^o$ are independent Brownian bridges, independent of the standard normal random variable $Z$. Moreover, if $\mathbb{F}$ is any $P_0-$Donsker class of functions, then the convergences hold in probability in $l^{\infty}(\mathbb{F})$. In this case, the convergences is also $P_0^{\infty}-$almost surely under an additional condition that $P_0||f-P_0f||_{\mathbb{F}}^2 < \infty$.
\end{theorem}

We refer to Theorem	2.11.1 and 2.11.9 in \citep{van1996} for more details of the discussion for $\mathbb{F}$ such that the convergence holds in $l^{\infty}(\mathbb{F})$. 

When $P_0$ is continuous, there is a ``bias" term $\sigma(H-\mathbb{P}_n)$ in the convergence in \eqref{eq1: BVM continuous}. And the term vanishes only when $\sigma=0$, under which $P$ becomes the Dirichlet process, or when $H=\mathbb{P}_n$ ($H= P_0$), which is unrealistic. Moreover, the $\sigma$ equals the $\bar{C}_1$ in Theorem \ref{consistency thm}. Thus, it suggests that one is not expected to use NGGP for continuous $P_0$.

On the other hand, it is interesting to see that there is a ``bias" term $\frac{\sigma n(\pi)}{n}(H-\tilde{\mathbb{P}}_n)$ on the left hand side of  the convergence in \eqref{eq1: BVM dicrete} when $P_0$ is discrete to make the limiting process is $\mathbb{B}_{P_0}^o$. We can not drop this ``bias" term directly, although $\lim_{n\rightarrow \infty} \frac{n(\pi)}{n} =0$ a.s.. The term can be dropped as long as $\lim_{n\rightarrow \infty} \frac{n(\pi)}{\sqrt{n}} =0$, in the sense that the number of  atoms $\{x_j\}$ in $P_0$ should decrease fast enough when $n \rightarrow \infty$. For a formal condition of $P_0$ to make $\lim_{n\rightarrow \infty} \frac{n(\pi)}{\sqrt{n}} =0$, we have the following corollary.
\begin{corollary}\label{cor: discrete true BVM}
Under the conditions in Theorem \ref{thm: BVM}, when $P_0$ is discrete, we have the following results.
\begin{itemize}
\item[(i)] If $P_0(\{x_j\})\leq \frac{C}{j^{\alpha}}$, for some positive constant $C$ and $\alpha>2$ and $\mathbb{F}$ is the class of uniformly bounded functions, then $\sqrt{n}(P_{U_n}-\mathbb{P}_n)|\bX \leadsto \mathbb{B}_{P_0}^o$ in probability in $l^{\infty}(\mathbb{F})$.
\item[(ii)] If the function $h(t):=\# \{x: P_0(\{x\})\geq \frac{1}{t}\}$ is regularly varying at $\infty$ of exponent $\eta$ with $\eta <\frac{1}{2}$ and $\mathbb{F}$ is the class of uniformly bounded functions, then $\sqrt{n}(P_{U_n}-\mathbb{P}_n)|\bX \leadsto \mathbb{B}_{P_0}^o$ a.s. in $l^{\infty}(\mathbb{F})$.
\item[(iii)] If $\mathbb{F}$ is a class of functions $f$ such that $f(\{x_j\})\asymp j^\beta
$ for some $\beta>0$ and $P_0(\{x_j\})\leq \frac{C}{j^{\alpha}}$, for some positive constant $C$ and $\alpha>2+2\beta$, then $\sqrt{n}(P_{U_n}-\mathbb{P}_n)|\bX \leadsto \mathbb{B}_{P_0}^o$ in probability in $l^{\infty}(\mathbb{F})$.
\end{itemize}
\end{corollary}

The proof of the above Corollary follows directly from the Corollary 2 in \citep{franssen2022bernstein}. And we recall that if $h$ is regularly varying at $\infty$ with exponent $\eta \in (0,1)$, then for any $t>0$, we have $\lim_{n \rightarrow \infty} \frac{h(nt)}{h(n)} =t^{\eta}$. Moreover, for such regularly varying function $h$, we have  $\frac{n(\pi)}{h(n)} \rightarrow \Gamma(1-\eta)$ a.s., and $h(n)$ is $n^{\eta}$ up to a slowly varying factor. We refer the appendix in \citep{haan2006extreme} and \citep{bingham_goldie_teugels_1987} for more details of the regularly varying function.

As the application of the Bernstein-von Mises results in \cref{thm: BVM}, we may construct Bayesian credible sets for $Pf$ when $n \rightarrow \infty$. The choices of $f$ determine the parameters, for which the credible sets are constructed. For example, if $f(x)=x$, the credible interval is for the mean.  Since the posterior consistency does not hold for the case when $P_0$ is continuous, the credible sets for $Pf$ is not correct in this case, thus we shall only give the credible sets for $Pf$ when $P_0$ is discrete.
\begin{corollary}\label{cor: confidence interval}
If $P_0$ is discrete, under the conditions in Theorem \ref{thm: BVM}, we have the probability of $P_0f  \in \lc L_{n, \alpha}f- \frac{\sigma n(\pi)}{n}(Hf-\tilde{\mathbb{P}}_nf),  L_{n,\beta}f -\frac{\sigma n(\pi)}{n}(Hf-\tilde{\mathbb{P}}_nf)\rc$ is $\beta-\alpha$ for any $f$ such that $P_0f^2<\infty$ and $Hf^2<\infty$. Here $L_{n,\alpha}$ is the $\alpha-$quantile of the posterior distribution of $Pf|\bX$ and $\beta>\alpha$.
\end{corollary}
One direct interpretation of the above corollary is one may want $\frac{n(\pi)}{n} \rightarrow 0$ in probability to make  the ``bias" term vanish and therefore the confidence interval for $P_0f$ becomes a regular form $ \lc L_{n, \alpha}f,  L_{n,\beta}f)\rc$. This is true under the case (i) of \cref{cor: discrete true BVM}, or when $f(x)=x$ with $\alpha>4$. Otherwise, the correction $\frac{\sigma n(\pi)}{n}(Hf-\tilde{\mathbb{P}}_nf)$ is necessary as a bias correction to the credible interval. We provide a numerical illustration  corresponding to this scenario in \cref{subsection: numerical}.

However, $P_0$ is of course unknown in the real application and we shall consider Theorem \ref{thm: BVM} without the information from $P_0$. In this case,  one needs to pay especial attention to the parameter $\sigma$, and it is easy to see from both  Theorem \ref{consistency thm} and  Theorem \ref{thm: BVM} that if $\sigma\rightarrow 0$, $P$ is posterior consistent and the Bernstein-von Mises results hold without the bias terms for any $P_0$. But this corresponds to the case that $P$ becomes the Dirichlet process. Thus, one should at least expect the parameter $\sigma$ to be small. Usually, the model parameters are chosen by the empirical Bayesian method, and people can estimate the model parameters by using the maximum likelihood estimators conditional on the observations $\bX$. A well known conclusion \citep{pitman2003poisson, pitman2006combinatorial} in Bayesian nonparametric framework is the   observation $\bX$ from NRMIs induces a random partition structure for $\{1,\cdots,n\}$ as we introduced in \cref{subsection: posterior of NRMIs}. The random partition structure is characterized by the exchangeable partition probability function (EPPF) \citep{pitman2003poisson}, which also plays the rule as the likelihood function of $\sigma$ as explained in e.g., \citep{favaro2021near, ghosal2017fundamentals, franssen2022bernstein}. And the EPPF for the NGGP is given as 
\begin{align*}
\Pi_{\sigma}(n_1,\cdots, n_{n(\pi)})=\frac{\prod_{j=1}^{n(\pi)} (1-\sigma)_{(n_j-1)}}{\Gamma(n)} \int_0^{\infty} u^{n-1}(u+\theta)^{n(\pi)\sigma-n}e^{\frac{a}{\sigma}\lc(u+\theta \rc^{\sigma}-\theta^{\sigma}}du\,,
\end{align*}
where $(1-\sigma)_{(n_j-1)}=\frac{\Gamma(n_j-\sigma)}{\Gamma(1-\sigma)}$. From Theorem 1 in \citep{favaro2021near}, the maximum likelihood estimator $\hat{\sigma}_n$ exists uniquely. Furthermore, the results in Theorem 2 in \citep{favaro2021near} implies that $\hat{\sigma}_n \rightarrow \sigma_0$ in probability with a rate $\sqrt{\log (n)}n^{-\frac{\sigma_0}{2}}$, when $P_0$ is discrete with atoms $x$ satisfying $h(t)=\#\{P_0(\{x\})\geq \frac{1}{t}\}$ is a regularly varying function of exponent $\sigma_0 \in [0,1)$.
\begin{theorem}
Under the assumptions in \cref{thm: BVM}, we have the following results.
\begin{itemize}
\item[(i)] If $\hat{\sigma}_n$ is an estimator based on $\bX$ that converges to $\sigma_0$ in probability, then the convergences in \cref{thm: BVM} hold in probability by replacing $\sigma_n$ by $\hat{\sigma}_n$  and replacing $\sigma$ by $\sigma_0$. In particular, this is true for the maximum likelihood estimator $\hat{\sigma}_n$, if $P_0$ is discrete with atoms $x$ satisfying the condition that $h(t)=\#\{P_0(\{x\})\geq \frac{1}{t}\}$ is a regularly varying function of exponent $\sigma_0 \in [0,1)$.
\item[(ii)] If $\sigma \sim L_{\sigma}$, where $L_{\sigma}$ is a probability law on $[0,1]$ that plays the prior distribution of $\sigma$, then the Bayesian model becomes
\begin{align*}
\bX |P, \sigma \sim P; \quad P|\sigma \sim \NGGP(a, \sigma, \theta, H)
\end{align*}
The convergences in \cref{thm: BVM} hold by replacing $\sigma_n$ by $\sigma$ on the left hand side, and replacing $\sigma$ by $\sigma_0$ on the limiting processes. The $\sigma$ on the left hand side is the posterior random variable.
\end{itemize}
\end{theorem}

The proof of the above theorem follows the same constructions as the proof in section 4.2 of \citep{franssen2022bernstein}. For the posterior consistency of $\hat{\sigma}_n$, we refer to the details with proofs in section 4.3 of \citep{franssen2022bernstein}. The maximum likelihood estimator is not quite interesting as $\hat{\sigma}_n \rightarrow \sigma_0$ with  $\sigma_0=1$ when $P_0$ is continuous, and $\sigma_0\neq 0$ when $P_0$ is discrete \citep{favaro2021near}.

Besides the parameter $\sigma$, the parameters $a$ and  $\theta$ don't appear in the asymptotic results in Theorem \ref{consistency thm} and Theorem \ref{thm: BVM}, and thus estimators of $a$ and  $\theta$ based on prior distributions or maximum likelihood method won't affect the convergences when $a<< \sqrt{n}$ and $\theta << n^{\sigma}$. And the cases when $\hat{a}_n$ and $\hat{\theta}_n$ converge to $\infty$ as $n \rightarrow \infty$ are not usual and beyond the scope of this work and can be considered in the future works.

\subsection{Numerical illustration}\label{subsection: numerical}

We present the credible intervals for $P_0f$ when $P_0$ is discrete with different types of the number of atoms. To be more precise, let $P_0f=P_0([2,\infty])$ for $P_0=P_1, P_2,P_3,P_4$, where we describe $P_1, P_2,P_3,P_4$ as follows. Let the probability distributions of $P_1, P_2,P_3,P_4$ be on $\mathbb{Z}^+$ are as follows.
\begin{align*}
&P_1(X=1)=0.2, P_1(X=2)=0.2,P_1(X=3)=0.2,P_1(X=4)=0.3,P_1(X=5)=0.1\,,\\
&P_2(X=k) \varpropto k^{-3}\,,\qquad P_3(X=k) \varpropto k^{-2}\,,\qquad P_4(X=k) \varpropto k^{-\frac{3}{2}}\,.
\end{align*}
Obviously, $n(\pi)=5$ for $P_1$. From the result (see e.g., Example 4) in \citep{karlin1967central}, we have the regularly varying functions $h(t)$ corresponding to $P_2,P_3,P_4$ are proportional to $t^{\frac{1}{3}}, t^{\frac{1}{2}}, t^{\frac{2}{3}}$ respectively. And when $n\rightarrow \infty$, the distinct numbers $n(\pi)$ of $P_2,P_3,P_4$ are proportional to $n^{\frac{1}{3}}, n^{\frac{1}{2}}, n^{\frac{2}{3}}$, respectively, from Theorem 1 in \citep{karlin1967central}. Thus, the ``bias" term for $P_1, P_2,P_3,P_4$ goes to $0,0$, some constant, $\infty$, respectively.

For the NGGP, we let $P \sim \NGGP(1,\sigma=0.5, 1, H)$, where $H$ is standard normal distribution. We simulate $P$ through its stick-breaking representation with the generating algorithm in \citep{favaro2016}. To make sure the simulation of $P=\sum_{i=1}^{\infty}w_i \delta_{X_i}$ is accurate, we truncate the infinite sum at some $N$ such that the weight of the tail $\sum_{i=N}^{\infty} w_i<\frac{1}{\sqrt{n}}$, where $n$ is the sample size. We simulate $10000$ replications of the sample $\bX$ from  $P_1, P_2,P_3,P_4$ with the sample size $n=10,100,1000, 10000,100000$ respectively. For the sample from $P_1$, we construct one $95\%$ credible interval for each sample for $P_1([2,\infty))$ with the ``bias" correction as in \cref{cor: confidence interval}  and compute the proportion that the true value $P_1([2,\infty))$ belongs to the intervals of 10000 replications. And we also compute the same proportion without the ``bias" correction.  The results of  $P_1, P_2,P_3,P_4$ are given in tables \ref{table1: bias cl} and \ref{table2: unbias cl}.
\begin{table}
\begin{center}
\begin{tabular}{c| c c c c c} 
 \hline
$n$ & $10$ & $100$ & $1000$ & $10000$ & $100000$  \\ 
\hline
 $P_1$ & $0.791$ & $0.952$ & $0.961$ & $0.967$ & $0.986$  \\
 $P_2$ & $0.695$ & $0.857$ & $0.928$ & $0.917$ & $0.931$  \\ 
 $P_3$ & $0.712$ & $0.785$ & $0.811$ & $0.727$ & $0.754$  \\
 $P_4$& $0.601$ & $0.292$ & $0.078$ & $0.000$ & $0.000$  \\
 \hline
\end{tabular}
\end{center}
\caption{\label{table1: bias cl} Proportion of coverage of the true value for the $95\%$ credible interval without ``bias" correction.}
\end{table}

\begin{table}
\begin{center}
\begin{tabular}{c| c c c c c} 
 \hline
$n$ & $10$ & $100$ & $1000$ & $10000$ & $100000$  \\ 
\hline
 $P_1$ & $0.977$ & $0.989$ & $0.991$ & $0.995$ & $0.997$  \\
 $P_2$ & $0.914$ & $0.938$ & $0.951$ & $0.933$ & $0.941$  \\ 
 $P_3$ & $0.863$ & $0.931$ & $0.962$ & $0.960$ & $0.978$  \\
 $P_4$& $0.901$ & $0.955$ & $0.969$ & $0.966$ & $0.956$  \\
 \hline
\end{tabular}
\end{center}
\caption{\label{table2: unbias cl} Proportion of coverage of the true value for the $95\%$ credible interval with ``bias" correction.}
\end{table}
Since the ``bias" terms for $P_1$ and $P_2$ vanish as $n \rightarrow \infty$, the proportions of the coverage of the true value are large for both with and without ``bias" correction. And the $95\%$ credible intervals for $P_3f$ and $P_4f$ are not performing good without ``bias" correction.

As for the normality convergence, we draw the marginal density plots in figure \ref{fig: density} for $P_1([2,\infty))$ given the sample $\bX$ with size $n=10, 100, 1000, 10000, 100000$ respectively. Both plots are generated from $1000000$ replicates, the true mean of $P_1([2,\infty))$ is $0.8$ The marginal density for $P_1([2,\infty))$ is skewed when $n=10, 100$, and symmetric when $n=1000$ and larger.
\begin{figure}[h]
  \includegraphics[width=0.32\linewidth , height=5cm]{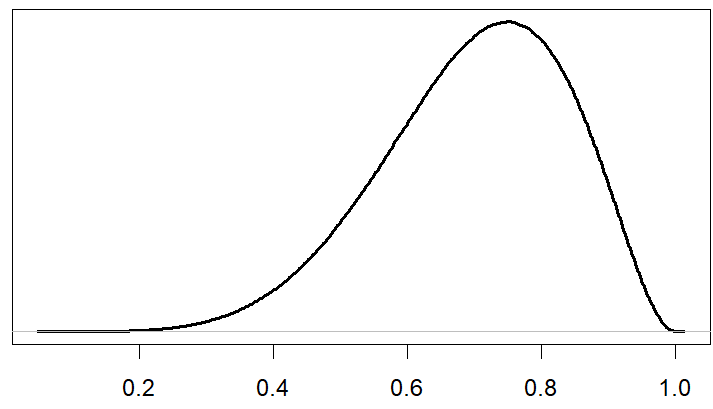}
  \includegraphics[width=0.32\linewidth , height=5cm]{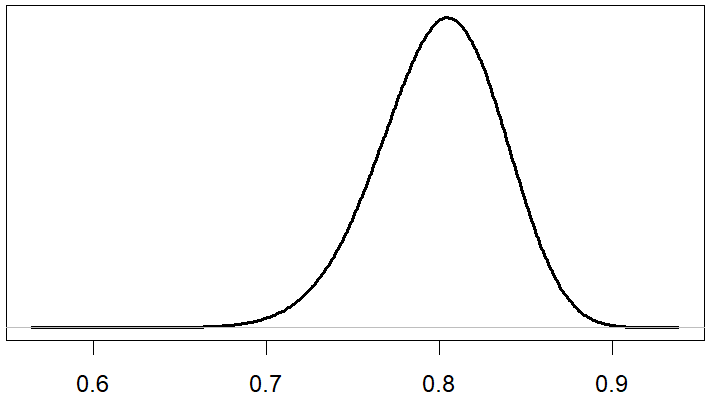}
  \includegraphics[width=0.32\linewidth , height=5cm]{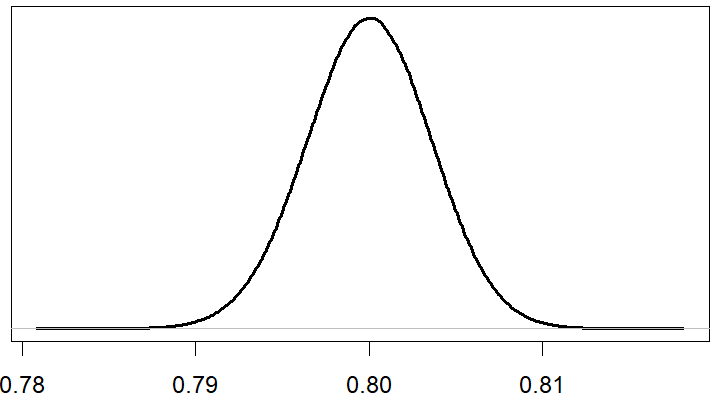}
  \includegraphics[width=0.33\linewidth , height=5cm]{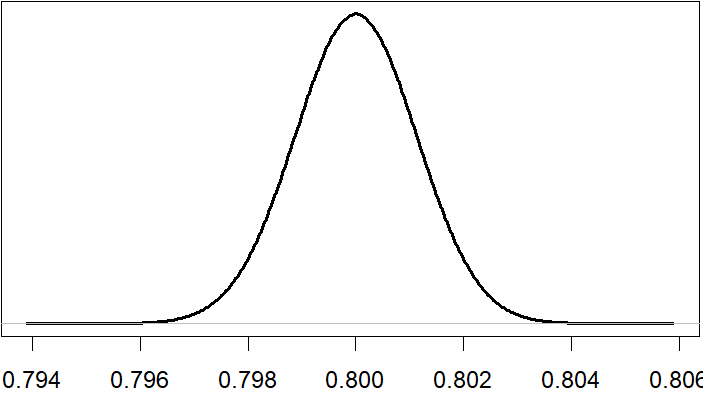}
  \includegraphics[width=0.33\linewidth , height=5cm]{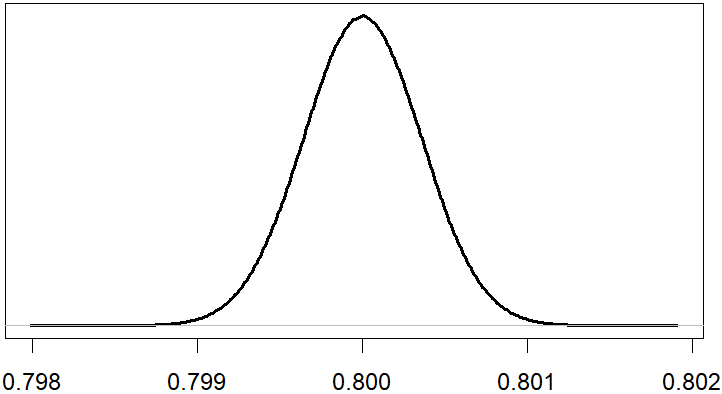}
\caption{The marginal densities for $P_1([2,\infty))$ with sample size $n=10,100,1000, 10000,100000$ follow the order from top left to bottom right.}
\label{fig: density}
\end{figure}
\section{Discussion}\label{sec.discussion}

To the best of  our knowledge, the L\'evy intensities of  the well-studied   NRMIs 
up-to-date are  given in the form of  the gamma  density:  $s^{-\sigma-1}e^{-\beta s}$.  It turns out that with the shape parameter $\sigma=0$, the posterior consistency is always   guaranteed for any ``true" prior  distribution $P_0$. Otherwise,  the posterior consistency only holds for discrete prior  $P_0$ but not   for continuous   $P_0$. Such phenomenon does naturally make sense due to the discreteness of NRMIs  (the completely random measures \citep{kingman1975}).  As explained in the Bayesian literature, if $P_0$ is diffusive  and the prior guess for the sample distribution $\alpha \neq P_0$, then the prior guess will always contribute to the posterior, no matter how large is the sample size. In such sense, the Bayesian nonparametric models never behave ``better" than the empirical models asymptomatically. However, this doesn't mean the NRMIs are not useful. On the one hand, we are not able to know the ``true" distribution of a given sample with any size $n$, also the sample size $n$ will never be $\infty$, a prior guess of the random probability measure based on experience could make the model suitable. On the other hand, the NRMIs behave great for the data from discrete distributions. Furthermore, the mixture and hierarchical Bayesian nonparametric models based on NRMIs are showing great success in the applications and consistency behaviours \citep{lijoi2005consistency}. And the class of NRMIs is much larger than we expected, so that more study is necessary to develop more flexible subclasses of NRMIs or more general NRMIs like classes that are satisfying the consistency property.  The results in this work provides a guideline of choosing the proper intensity $\rho(s|x)$, for example, the generalized Dirchlet process and the generalized extended gamma NRMI are good choice in the Bayesian nonparametric applications and they both show some flexibility. Besides, we may let $\sigma \rightarrow 0$ by assigning a randomness on $\sigma$, or one may construct $\alpha$ to depend on $\rho(ds|x)$ to deduct $\bar{C}_1$.

Due to the complexity of the posterior of the NRMIs, it is not easy to present a Bernstein-von Mises like result to give the limiting process of posterior of general NRMIs. The result for the normalized generalized gamma process, along with the works in \citep{lo1983,lo1986remark, ray2021bernstein, huzhang2022functional,james2008, franssen2022bernstein}, shed some light in discovering the Bernstein-von Mises theorem for general NRMIs.
\section*{Acknowledgments}
\bibliographystyle{Chicago}
\bibliography{huzhang}

\begin{thebibliography}{}

\bibitem[\protect\citeauthoryear{Aldous, Ibragimov, Jacod, and Aldous}{Aldous
  et~al.}{1985}]{aldous1985exchangeability}
Aldous, D.~J., I.~A. Ibragimov, J.~Jacod, and D.~J. Aldous (1985).
\newblock {\em Exchangeability and related topics}.
\newblock Springer.

\bibitem[\protect\citeauthoryear{Antoniak}{Antoniak}{1974}]{antoniak1974mixtures}
Antoniak, C.~E. (1974).
\newblock Mixtures of {Dirichlet} processes with applications to {Bayesian}
  nonparametric problems.
\newblock {\em The Annals of Statistics\/}, 1152--1174.

\bibitem[\protect\citeauthoryear{Bingham, Goldie, and Teugels}{Bingham
  et~al.}{1987}]{bingham_goldie_teugels_1987}
Bingham, N.~H., C.~M. Goldie, and J.~L. Teugels (1987).
\newblock {\em Regular Variation}.
\newblock Encyclopedia of Mathematics and its Applications. Cambridge:
  Cambridge University Press.

\bibitem[\protect\citeauthoryear{Daley and Vere-Jones}{Daley and
  Vere-Jones}{2008}]{dalay2008book}
Daley, D. and D.~Vere-Jones (2008).
\newblock {\em An Introduction to the Theory of Point Processes: Volume II:
  General Theory and structure\/} (Second ed.).
\newblock Probability and its Applications. Springer, New York.

\bibitem[\protect\citeauthoryear{{De Blasi}, Lijoi, and Pruenster}{{De Blasi}
  et~al.}{2013}]{Blasi_2013}
{De Blasi}, P., A.~Lijoi, and I.~Pruenster (2013).
\newblock An asymptotic analysis of a class of discrete nonparametric priors.
\newblock {\em Statistica Sinica\/}~{\em 23\/}(3), 1299--1322.

\bibitem[\protect\citeauthoryear{Dykstra and Laud}{Dykstra and
  Laud}{1981}]{dykstra1981bayesian}
Dykstra, R. and P.~Laud (1981).
\newblock A {Bayesian} nonparametric approach to reliability.
\newblock {\em The Annals of Statistics\/}~{\em 9\/}(2), 356--367.

\bibitem[\protect\citeauthoryear{Favaro, Lijoi, Nava, Nipoti, Pr{\"u}nster, and
  Teh}{Favaro et~al.}{2016}]{favaro2016}
Favaro, S., A.~Lijoi, C.~Nava, B.~Nipoti, I.~Pr{\"u}nster, and Y.~W. Teh (2016,
  sep).
\newblock On the stick-breaking representation for homogeneous {NRMIs}.
\newblock {\em Bayesian Analysis\/}~{\em 11\/}(3), 697--724.

\bibitem[\protect\citeauthoryear{Favaro and Naulet}{Favaro and
  Naulet}{2021}]{favaro2021near}
Favaro, S. and Z.~Naulet (2021).
\newblock Near-optimal estimation of the unseen under regularly varying tail
  populations.
\newblock {\em arXiv preprint arXiv:2104.03251\/}.

\bibitem[\protect\citeauthoryear{Ferguson}{Ferguson}{1973}]{ferguson1973}
Ferguson, T.~S. (1973, mar).
\newblock A {Bayesian} analysis of some nonparametric problems.
\newblock {\em The Annals of Statistics\/}~{\em 1\/}(2), 209--230.

\bibitem[\protect\citeauthoryear{Franssen and van~der Vaart}{Franssen and
  van~der Vaart}{2022}]{franssen2022bernstein}
Franssen, S. and A.~van~der Vaart (2022).
\newblock {Bernstein-von Mises} theorem for the {Pitman-Yor} process of
  nonnegative type.
\newblock {\em Electronic Journal of Statistics\/}~{\em 16\/}(2), 5779--5811.

\bibitem[\protect\citeauthoryear{Freedman and Diaconis}{Freedman and
  Diaconis}{1983}]{freedman1983inconsistent}
Freedman, D. and P.~Diaconis (1983).
\newblock On inconsistent {Bayes} estimates in the discrete case.
\newblock {\em The Annals of Statistics\/}, 1109--1118.

\bibitem[\protect\citeauthoryear{Ghosal and Van~der Vaart}{Ghosal and Van~der
  Vaart}{2017}]{ghosal2017fundamentals}
Ghosal, S. and A.~Van~der Vaart (2017).
\newblock {\em Fundamentals of nonparametric {Bayesian} inference\/} (First
  ed.), Volume~44.
\newblock Cambridge: Cambridge University Press.

\bibitem[\protect\citeauthoryear{Gnedin and Pitman}{Gnedin and
  Pitman}{2006}]{gnedin2006exchangeable}
Gnedin, A. and J.~Pitman (2006).
\newblock Exchangeable {G}ibbs partitions and stirling triangles.
\newblock {\em Journal of Mathematical sciences\/}~{\em 138}, 5674--5685.

\bibitem[\protect\citeauthoryear{Haan and Ferreira}{Haan and
  Ferreira}{2006}]{haan2006extreme}
Haan, L. and A.~Ferreira (2006).
\newblock {\em Extreme value theory: an introduction}, Volume~3.
\newblock Springer.

\bibitem[\protect\citeauthoryear{Ho~Jang, Lee, and Lee}{Ho~Jang
  et~al.}{2010}]{jang2010}
Ho~Jang, G., J.~Lee, and S.~Lee (2010).
\newblock Posterior consistency of species sampling priors.
\newblock {\em Statistica Sinica\/}~{\em 20\/}(2), 581--593.

\bibitem[\protect\citeauthoryear{Hu and Zhang}{Hu and
  Zhang}{2022}]{huzhang2022functional}
Hu, Y. and J.~Zhang (2022).
\newblock Functional central limit theorems for stick-breaking priors.
\newblock {\em Bayesian Analysis\/}~{\em 17\/}(4), 1101--1120.

\bibitem[\protect\citeauthoryear{James}{James}{2002}]{james2002poisson}
James, L.~F. (2002).
\newblock Poisson process partition calculus with applications to exchangeable
  models and {Bayesian} nonparametrics.
\newblock {\em arXiv preprint math/0205093\/}.

\bibitem[\protect\citeauthoryear{James}{James}{2008}]{james2008}
James, L.~F. (2008).
\newblock Large sample asymptotics for the two-parameter
  {Poisson{\textendash}Dirichlet} process.
\newblock In {\em In Clarke, B. and Ghosal, S. (eds.), Pushing the Limits of
  Contemporary Statistics: Contributions in Honor of Jayanta K. Ghosh,},
  Volume~3, Cambridge, MA, pp.\  187--199. MIT Press.

\bibitem[\protect\citeauthoryear{James, Lijoi, and Pr{\"u}nster}{James
  et~al.}{2006}]{james2006conjugacy}
James, L.~F., A.~Lijoi, and I.~Pr{\"u}nster (2006).
\newblock Conjugacy as a distinctive feature of the {D}irichlet process.
\newblock {\em Scandinavian Journal of Statistics\/}~{\em 33\/}(1), 105--120.

\bibitem[\protect\citeauthoryear{James, Lijoi, and Pr{\"u}nster}{James
  et~al.}{2009}]{james2009}
James, L.~F., A.~Lijoi, and I.~Pr{\"u}nster (2009, jul).
\newblock Posterior analysis for normalized random measures with independent
  increments.
\newblock {\em Scandinavian Journal of Statistics\/}~{\em 36\/}(1), 76--97.

\bibitem[\protect\citeauthoryear{Karlin}{Karlin}{1967}]{karlin1967central}
Karlin, S. (1967).
\newblock Central limit theorems for certain infinite urn schemes.
\newblock {\em Journal of Mathematics and Mechanics\/}~{\em 17\/}(4), 373--401.

\bibitem[\protect\citeauthoryear{Kingman}{Kingman}{1967}]{kingman1967completely}
Kingman, J. (1967).
\newblock Completely random measures.
\newblock {\em Pacific Journal of Mathematics\/}~{\em 21\/}(1), 59--78.

\bibitem[\protect\citeauthoryear{Kingman}{Kingman}{1975}]{kingman1975}
Kingman, J.~F. (1975).
\newblock Random discrete distributions.
\newblock {\em Journal of the Royal Statistical Society: Series B
  (Methodological)\/}~{\em 37\/}(1), 1--15.

\bibitem[\protect\citeauthoryear{Kingman}{Kingman}{1993}]{kingman1993poisson}
Kingman, J. F.~C. (1993).
\newblock {\em Poisson processes}, Volume~3 of {\em Oxford Studies in
  Probability}.
\newblock Clarendon Press, Oxford University Press, New York.

\bibitem[\protect\citeauthoryear{Korwar and Hollander}{Korwar and
  Hollander}{1973}]{Korwar_1973}
Korwar, R.~M. and M.~Hollander (1973, aug).
\newblock Contributions to the theory of {Dirichlet} processes.
\newblock {\em The Annals of Probability\/}~{\em 1\/}(4).

\bibitem[\protect\citeauthoryear{Lijoi, Mena, and Pr{\"u}nster}{Lijoi
  et~al.}{2005a}]{lijoi2005a}
Lijoi, A., R.~H. Mena, and I.~Pr{\"u}nster (2005a, dec).
\newblock Bayesian nonparametric analysis for a generalized {D}irichlet process
  prior.
\newblock {\em Statistical Inference for Stochastic Processes\/}~{\em 8\/}(3),
  283--309.

\bibitem[\protect\citeauthoryear{Lijoi, Mena, and Pr{\"u}nster}{Lijoi
  et~al.}{2005b}]{lijoi2005b}
Lijoi, A., R.~H. Mena, and I.~Pr{\"u}nster (2005b, dec).
\newblock Hierarchical mixture modeling with normalized inverse-{G}aussian
  priors.
\newblock {\em Journal of the American Statistical Association\/}~{\em
  100\/}(472), 1278--1291.

\bibitem[\protect\citeauthoryear{Lijoi, Mena, and Pr{\"u}nster}{Lijoi
  et~al.}{2007}]{lijoi_2007}
Lijoi, A., R.~H. Mena, and I.~Pr{\"u}nster (2007, sep).
\newblock Controlling the reinforcement in {Bayesian} non-parametric mixture
  models.
\newblock {\em Journal of the Royal Statistical Society: Series B (Statistical
  Methodology)\/}~{\em 69\/}(4), 715--740.

\bibitem[\protect\citeauthoryear{Lijoi, Pr{\"u}nster, et~al.}{Lijoi
  et~al.}{2003}]{lijoi2003normalized}
Lijoi, A., I.~Pr{\"u}nster, et~al. (2003).
\newblock On a normalized random measure with independent increments relevant
  to {Bayesian} nonparametric inference.
\newblock In {\em 13th European Young Statisticians Meeting}, pp.\  123--134.
  Staempfli.

\bibitem[\protect\citeauthoryear{Lijoi, Pr{\"u}nster, et~al.}{Lijoi
  et~al.}{2010}]{lijoi2010models}
Lijoi, A., I.~Pr{\"u}nster, et~al. (2010).
\newblock Models beyond the {D}irichlet process.
\newblock {\em Bayesian nonparametrics\/}~{\em 28\/}(80), 342.

\bibitem[\protect\citeauthoryear{Lijoi, Pr{\"u}nster, and Walker}{Lijoi
  et~al.}{2005}]{lijoi2005consistency}
Lijoi, A., I.~Pr{\"u}nster, and S.~G. Walker (2005).
\newblock On consistency of nonparametric normal mixtures for {Bayesian}
  density estimation.
\newblock {\em Journal of the American Statistical Association\/}~{\em
  100\/}(472), 1292--1296.

\bibitem[\protect\citeauthoryear{Lo}{Lo}{1982}]{lo1982bayesian}
Lo, A.~Y. (1982).
\newblock Bayesian nonparametric statistical inference for {Poisson} point
  processes.
\newblock {\em Zeitschrift f{\"u}r Wahrscheinlichkeitstheorie und verwandte
  Gebiete\/}~{\em 59\/}(1), 55--66.

\bibitem[\protect\citeauthoryear{Lo}{Lo}{1983}]{lo1983}
Lo, A.~Y. (1983).
\newblock Weak convergence for {Dirichlet} processes.
\newblock {\em Sankhy{\=a}: The Indian Journal of Statistics, Series A\/},
  105--111.

\bibitem[\protect\citeauthoryear{Lo}{Lo}{1986}]{lo1986remark}
Lo, A.~Y. (1986).
\newblock A remark on the limiting posterior distribution of the multiparameter
  {D}irichlet process.
\newblock {\em Sankhy{\=a}: The Indian Journal of Statistics, Series A\/},
  247--249.

\bibitem[\protect\citeauthoryear{M{\"u}ller and Quintana}{M{\"u}ller and
  Quintana}{2004}]{muller2004nonparametric}
M{\"u}ller, P. and F.~A. Quintana (2004).
\newblock Nonparametric {Bayesian} data analysis.
\newblock {\em Statistical Science\/}~{\em 19\/}(1), 95--110.

\bibitem[\protect\citeauthoryear{Perman, Pitman, and Yor}{Perman
  et~al.}{1992}]{perman1992}
Perman, M., J.~Pitman, and M.~Yor (1992, mar).
\newblock Size-biased sampling of {P}oisson point processes and excursions.
\newblock {\em Probability Theory and Related Fields\/}~{\em 92\/}(1), 21--39.

\bibitem[\protect\citeauthoryear{Pitman}{Pitman}{1996}]{pitman1996some}
Pitman, J. (1996).
\newblock Some developments of the {Blackwell-MacQueen} urn scheme.
\newblock In {\em Statistics, Probability and Game Theory, Lecture
  Notes-Monograph Series}, Volume~30, Hayward, CA, pp.\  245--267.

\bibitem[\protect\citeauthoryear{Pitman}{Pitman}{2003}]{pitman2003poisson}
Pitman, J. (2003).
\newblock {Poisson-Kingman} partitions.
\newblock In {\em Statistics and Science: a Festschrift for Terry Speed,
  Lecture Notes-Monograph Series}, Volume~40, Beachwood, OH, pp.\  1--34.

\bibitem[\protect\citeauthoryear{Pitman}{Pitman}{2006}]{pitman2006combinatorial}
Pitman, J. (2006).
\newblock {\em Combinatorial stochastic processes: Ecole d'et{\'e} de
  probabilit{\'e}s de saint-flour xxxii-2002}.
\newblock Springer.

\bibitem[\protect\citeauthoryear{Pitman and Yor}{Pitman and
  Yor}{1997}]{pitman1997}
Pitman, J. and M.~Yor (1997, apr).
\newblock The two-parameter {Poisson-Dirichlet} distribution derived from a
  stable subordinator.
\newblock {\em The Annals of Probability\/}~{\em 25\/}(2), 855--900.

\bibitem[\protect\citeauthoryear{Pr{\ae}stgaard and Wellner}{Pr{\ae}stgaard and
  Wellner}{1993}]{praestgaard1993exchangeably}
Pr{\ae}stgaard, J. and J.~A. Wellner (1993).
\newblock Exchangeably weighted bootstraps of the general empirical process.
\newblock {\em The Annals of Probability\/}, 2053--2086.

\bibitem[\protect\citeauthoryear{Ray and van~der Vaart}{Ray and van~der
  Vaart}{2021}]{ray2021bernstein}
Ray, K. and A.~van~der Vaart (2021).
\newblock On the {Bernstein-von Mises} theorem for the {D}irichlet process.
\newblock {\em Electronic Journal of Statistics\/}~{\em 15}, 2224--2246.

\bibitem[\protect\citeauthoryear{Regazzini, Guglielmi, Di~Nunno,
  et~al.}{Regazzini et~al.}{2002}]{regazzini2002}
Regazzini, E., A.~Guglielmi, G.~Di~Nunno, et~al. (2002).
\newblock Theory and numerical analysis for exact distributions of functionals
  of a {D}irichlet process.
\newblock {\em The Annals of Statistics\/}~{\em 30\/}(5), 1376--1411.

\bibitem[\protect\citeauthoryear{Regazzini, Lijoi, and Pr{\"u}nster}{Regazzini
  et~al.}{2003}]{Regazzini_2003}
Regazzini, E., A.~Lijoi, and I.~Pr{\"u}nster (2003, apr).
\newblock Distributional results for means of normalized random measures with
  independent increments.
\newblock {\em The Annals of Statistics\/}~{\em 31\/}(2), 560--585.

\bibitem[\protect\citeauthoryear{Vaart}{Vaart}{1998}]{vaart_1998}
Vaart, A. W. v.~d. (1998).
\newblock {\em Asymptotic Statistics}.
\newblock Cambridge Series in Statistical and Probabilistic Mathematics.
  Cambridge: Cambridge University Press.

\bibitem[\protect\citeauthoryear{van~der Vaart and Wellner}{van~der Vaart and
  Wellner}{1996}]{van1996}
van~der Vaart, A.~W. and J.~A. Wellner (1996).
\newblock {\em Weak Convergence and Empirical Processes with Applications to
  Statistics}.
\newblock New York: Springer.

\bibitem[\protect\citeauthoryear{Zhang and Hu}{Zhang and
  Hu}{2021}]{huzhangreview}
Zhang, J. and Y.~Hu (2021).
\newblock Dirichlet process and {Bayesian} nonparametric models (in chinese).
\newblock {\em SCIENTIA SINICA Mathematica\/}~{\em 51\/}(11), 1895--1932.

\end{thebibliography}
\section{Supplementary Materials}\label{supplementary}
%\subsection*{1.1  Intermediate Importance Sampling}
In this section, we prove
 \cref{them. moment},\cref{consistency thm} and \cref{thm: BVM}.
\subsection*{Proof of \cref{them. moment}}
Let $\mathcal{I}=\mathbb{E}[(\tP(A)|\bX)^m]$.
% for any positive integer $m$ and any $A \in \mathcal{X}$. 
 Then, by  Theorem \ref{posterior}, $\mathcal{I}$ can be computed as follows.
\begin{align}
&\mathcal{I}= \int_0^{\infty} \E[(\tP(A)|U_n = u, \bX)^m]f_{U_n|\bX}(u|\bX)du\nonumber\\
&=\int_0^{\infty} \E\left[\left(\frac{\tilde{\mu}_{(U_n)}(A)}{T_{(U_n)} + \sum_{j=1}^{n(\pi)}J_j} + \sum_{j=1}^{n(\pi)} \frac{J_j \delta_{Y_j}(A)}{T_{(U_n)} + \sum_{j=1}^{n(\pi)}J_j}\right)^m\right]f_{U_n|\bX}(u|\bX)du\nonumber\\
&=\sum_{k=0}^m {m \choose k} \frac{1}{\Gamma(m)}\int_{0}^{\infty} \int_0^{\infty} y^{m-1} \E \Bigg[e^{ -y(T_{(U_n)} + \sum_{j=1}^{n(\pi)}J_j)}\tilde{\mu}_{(U_n)}(A)^{m-k}  \nonumber \\
&\qquad \qquad \qquad \times \left(\sum_{j=1}^{n(\pi)}J_j \delta_{Y_j}(A)\right)^k\Bigg] f_{U_n|\bX}(u|\bX)dydu\,. \label{proof1.1}
\end{align}
Noticing  that $T_{(U_n)} =\tilde{\mu}_{(U_n)}(A)+\tilde{\mu}_{(U_n)}(A^c)$, where $\tilde{\mu}_{(U_n)}(A)$ and $\tilde{\mu}_{(U_n)}(A^c)$ are independent, we can rewrite  the expectation in \eqref{proof1.1} as
\begin{align}
 \mathcal{I}=&\sum_{k=0}^m {m \choose k} \frac{1}{\Gamma(m)}\int_{0}^{\infty} \int_0^{\infty} y^{m-1} \E \left[e^{- y \tilde{\mu}_{(U_n)}(A)}\tilde{\mu}_{(U_n)}(A)^{m-k}\right] \E \left[e^{ -y \tilde{\mu}_{(U_n)}(A^c)}\right]\nonumber\\
& \qquad \qquad \qquad  \E\left[e^{-y( \sum_{j=1}^{n(\pi)}J_j)}\left(\sum_{j=1}^{n(\pi)}J_j \delta_{Y_j}(A)\right)^k\right] f_{U_n|\bX}(u|\bX)dydu\nonumber\\
 =&\sum_{k=0}^m {m \choose k} \frac{1}{\Gamma(m)}\int_{0}^{\infty} \int_0^{\infty} y^{m-1} (-1)^{m-k}\E \left[\frac{d^{m-k}}{dy^{m-k}}e^{- y \tilde{\mu}_{(U_n)}(A)}\right] \E \left[e^{ -y \tilde{\mu}_{(U_n)}(A^c)}\right]\nonumber\\
&  \qquad \qquad \qquad  \left(\sum {k \choose l_1, \cdots, l_{n(\pi)}}\E\left[e^{-y( \sum_{j=1}^{n(\pi)}J_j)}\prod_{j=1}^{n(\pi)}J_j^{l_j} \delta_{Y_j}(A)\right]\right) f_{U_n|\bX}(u|\bX)dydu\nonumber\\
 =&\sum_{k=0}^m {m \choose k} \frac{1}{\Gamma(m)}\int_{0}^{\infty} \int_0^{\infty} y^{m-1} (-1)^{m-k}\E [\frac{d^{m-k}}{dy^{m-k}}e^{- y \tilde{\mu}_{(U_n)}(A)}] \E [e^{ -y \tilde{\mu}_{(U_n)}(A^c)}]\nonumber\\
& \qquad \qquad \qquad \left(\sum {k \choose l_1, \cdots, l_{n(\pi)}}\E[\prod_{j=1}^{n(\pi)}(-1)^{l_j}\frac{d^{l_j}}{d^{y^{l_j}}}e^{-yJ_j}\delta_{Y_j}(A)]\right) f_{U_n|\bX}(u|\bX)dydu\,,\nonumber
\end{align}
where the sum in front of ${k \choose l_1, \cdots, l_{n(\pi)}}$ is   over all the vector $(l_1, \cdots, l_{n(\pi)})$ such that $\sum_{j=1}^{n(\pi)}l_j=k$. 
Taking the derivatives inside the expectation and using  the Laplace transform of $\tilde{\mu}_{(U_n)}(A)$, we have 
\begin{align}
 \mathcal{I}
% =\sum_{k=0}^m {m \choose k} & \frac{1}{\Gamma(m)}\int_{0}^{\infty} \int_0^{\infty} y^{m-1} (-1)^{m-k}\left(\frac{d^{m-k}}{dy^{m-k}}\E \left[e^{- y \tilde{\mu}_{(U_n)}(A)}\right]\right) \E \left[e^{ -y\tilde{\mu}_{(U_n)}(A^c)}\right]\nonumber\\
%&\left(\sum {k \choose l_1, \cdots, l_{n(\pi)}}\prod_{j=1}^{n(\pi)}(-1)^{l_j}\frac{d^{l_j}}{d{y^{l_j}}}\E[e^{-yJ_j}]\delta_{Y_j}(A)\right) f_{U_n|\bX}(u|\bX)dydu\nonumber\\
 =\sum_{k=0}^m {m \choose k} &\frac{1}{\Gamma(m)}\int_{0}^{\infty} \int_0^{\infty} y^{m-1} (-1)^{m-k}\frac{d^{m-k}}{dy^{m-k}}\E [e^{- y \tilde{\mu}_{(U_n)}(A)}] \E [e^{ -y \tilde{\mu}_{(U_n)}(A^c)}]\nonumber \\
&\left(\sum {k \choose l_1, \cdots, l_{n(\pi)}}\prod_{j=1}^{n(\pi)}\frac{\tau_{n_j+l_j}(u+y,Y_j)}{\tau_{n_j}(u,Y_j)}\delta_{Y_j}(A)\right) f_{U_n|\bX}(u|\bX)dydu\nonumber\\
 =\sum_{k=0}^m {m \choose k}& \frac{1}{\Gamma(m)}\int_{0}^{\infty} \int_0^{\infty} y^{m-1} V_{\alpha(A)}^{(m-k)}(u,y) e^{-\psi_{\mathbb{X}}(u,y)}\nonumber\\
&\left(\sum {k \choose l_1, \cdots, l_{n(\pi)}}\prod_{j=1}^{n(\pi)}\frac{\tau_{n_j+l_j}(u+y,Y_j)}{\tau_{n_j}(u,Y_j)}\delta_{Y_j}(A)\right) f_{U_n|\bX}(u|\bX)dydu\,,\label{proof1.2}
\end{align}
where $\psi_{\mathbb{X}}(u,y)=\int_{\mathbb{X}}\int_0^{\infty}(1-e^{-ys})e^{-us}\rho(ds|x)\alpha(dx)$.
By the fact that $$f_{U_n|\bX}(u|\bX) \propto u^{n-1}e^{-\psi_{\mathbb{X}}(u)}\prod_{j=1}^{n(\pi)}\tau_{n_j}(u,Y_j)$$ and  $e^{-\psi_{\mathbb{X}}(u)}e^{-\psi_{\mathbb{X}}(u,y)}=e^{-\psi_{\mathbb{X}}(u+y)}$, we further simplify \eqref{proof1.2}  to 
\begin{align}
\mathcal{I}=\sum_{k=0}^m {m \choose k} & \frac{1}{\Gamma(m)}\int_{0}^{\infty} \int_0^{\infty} y^{m-1}u^{n-1} V_{\alpha(A)}^{(m-k)}(u+y) e^{-\psi_{\mathbb{X}}(u+y)}\nonumber\\
&\left(\sum {k \choose l_1, \cdots, l_{n(\pi)}}\prod_{j=1}^{n(\pi)}\tau_{n_j+l_j}(u+y,Y_j)\delta_{Y_j}(A)\right)dydu\nonumber\\
 =\sum_{k=0}^m {m \choose k} &\frac{1}{\Gamma(m)}\int_{0}^{\infty} \int_0^{\infty} y^{m-1}u^{n-1} V_{\alpha(A)}^{(m-k)}(u+y) e^{-\psi_{\mathbb{X}}(u+y)}\prod_{j=1}^{n(\pi)}\tau_{n_j}(u+y,Y_j)\nonumber\\
&\left(\sum {k \choose l_1, \cdots, l_{n(\pi)}}\prod_{j=1}^{n(\pi)}\frac{\tau_{n_j+l_j}(u+y,Y_j)}{\tau_{n_j}(u+y,Y_j)}\delta_{Y_j}(A)\right)dydu\,.\nonumber
\end{align}
The change of variable $(w,z)=(u+y,u)$ yields
\begin{align}
\mathcal{I} =\sum_{k=0}^m {m \choose k} &\frac{1}{\Gamma(m)}\int_{0}^{\infty} \int_0^{w} (w-z)^{m-1}z^{n-1} V_{\alpha(A)}^{(m-k)}(w) e^{-\psi_{\mathbb{X}}(w)}\prod_{j=1}^{n(\pi)}\tau_{n_j}(w,Y_j)\nonumber\\
&\left(\sum {k \choose l_1, \cdots, l_{n(\pi)}}\prod_{j=1}^{n(\pi)}\frac{\tau_{n_j+l_j}(w,Y_j)}{\tau_{n_j}(w,Y_j)}\delta_{Y_j}(A)\right)dzdw\,.\nonumber
\end{align}
Using 
\begin{align}
\int_0^{w} (w-z)^{m-1}z^{n-1}dz
%=w^{n+m-1}\int_0^1 (1-t)^{m-1}t^{n-1}dt
=w^{m+n-1}\frac{\Gamma(m)\Gamma(n)}{\Gamma(m+n)}\,,\nonumber
\end{align}
we obtain  
\begin{align}
 \mathcal{I}=&\sum_{k=0}^m {m \choose k}  \frac{1}{\Gamma(m)}\frac{\Gamma(m)\Gamma(n)}{\Gamma(m+n)}\int_{0}^{\infty}  w^{n+m-1} V_{\alpha(A)}^{(m-k)}(w) e^{-\psi_{\mathbb{X}}(w)}\prod_{j=1}^{n(\pi)}\tau_{n_j}(w,Y_j)\nonumber\\
&\qquad\qquad\qquad \left(\sum {k \choose l_1, \cdots, l_{n(\pi)}}\prod_{j=1}^{n(\pi)}\frac{\tau_{n_j+l_j}(w,Y_j)}{\tau_{n_j}(w,Y_j)}\delta_{Y_j}(A)\right)dw\nonumber\\
=&\frac{\Gamma(n)}{\Gamma(m+n)}\sum_{k=0}^m {m \choose k}
  \int_{0}^{\infty}  w^{m}f_{U_n}(w) V_{\alpha(A)}^{(m-k)}(w) \nonumber\\
& \qquad\qquad\qquad \left(\sum {k \choose l_1, \cdots, l_{n(\pi)}}\prod_{j=1}^{n(\pi)}\frac{\tau_{n_j+l_j}(w,Y_j)}{\tau_{n_j}(w,Y_j)}\delta_{Y_j}(A)\right)dw\nonumber\\
 =&\frac{\Gamma(n)}{\Gamma(m+n)}\sum_{0 \leq l_1+\cdots +l_{n(\pi)} \leq m}^m {m \choose l_1, \cdots, l_{n(\pi)}}  \int_{0}^{\infty}  w^{m}f_{U_n}(w) V_{\alpha(A)}^{(m-(l_1+\cdots +l_{n(\pi)}))}(w)\nonumber\\ 
&  \qquad\qquad\qquad \left(\prod_{j=1}^{n(\pi)}\frac{\tau_{n_j+l_j}(w,Y_j)}{\tau_{n_j}(w,Y_j)}\delta_{Y_j}(A)\right)dw\,.\nonumber
\end{align}
This is \eqref{moments}.  

For any   family of pairwise disjoint sets  $\{A_1, \cdots, A_q\}$ in $\mathcal{X}$ and   for any  positive integers $\{m_1,\cdots,
 m_q\}$  we denote  $A_{q+1}=(\cup_{i=1}^q A_i)^c$, $m_{q+1}=0$, and    $m=\sum_{i=1}^q m_i $.
 For any  sample $\{X_i\}_{i=1}^n$ from $P$, let $\{Y_j\}_{j=1}^{n(\pi)}$  be the distinct values  of $\{X_i\}_{i=1}^n$.  Let $\lambda_i=\{j: Y_j \in A_i\}$   be  the set of the index of $Y_j$'s that in $A_i$ and  we denote by $\#(\lambda_i)$ the number of components in $\lambda_i$. We can  compute the following moments easily.
\begin{align}
 \mathcal{L}:=&\E\left[ \tP(A_1)^{m_1} \cdots \tP(A_q)^{m_q}|\bX \right] \\
 =&\int_0^{\infty}\E\left[ \tP(A_1)^{m_1} \cdots \tP(A_q)^{m_q}|U_n=u, \bX \right]f_{U_n|\bX}(u|\bX)du\nonumber\\
=& \frac{1}{\Gamma(m)}\int_{0}^{\infty} \int_0^{\infty} y^{m-1} \E \left[e^{ -y(T_{u} + \sum_{j=1}^{n(\pi)}J_j)}\prod_{i=1}^{q+1} \left(\mu_u(A_i)+\sum_{j=1}^{n(\pi)}J_j \delta_{Y_j}(A_i)\right)^{m_i}\right]f_{U_n|\bX}(u|\bX)dydu\nonumber\\
=& \frac{1}{\Gamma(m)}\int_{0}^{\infty} \int_0^{\infty} y^{m-1} \prod_{i=1}^{q+1} \Bigg\{ \sum_{k=0}^{m_i} {m_i \choose k } \E \left[e^{ -y\mu_u(A_i)}\mu_u(A_i)^{m_i-k}\right]\nonumber\\
&\qquad\qquad\qquad \E \left[e^{ -y(\sum_{j\in \lambda_i}J_j)}\left(\sum_{j\in \lambda_i}J_j\right)^k\right]\Bigg\}  f_{U_n|\bX}(u|\bX)dydu\,. \nonumber
\end{align}
A similar  computation as that for  $\mathcal{I}$   yields 
\begin{align}
 \mathcal{L}=&\frac{\Gamma(n)}{\Gamma(m+n)}\int_{0}^{\infty}  w^{n+m-1}  e^{-\psi_{\mathbb{X}}(w)}\prod_{j=1}^{n(\pi)}\tau_{n_j}(w,Y_j) \prod_{i=1}^{q+1} \Bigg\{\sum_{0 \leq l_1+\cdots +l_{\#(\lambda_i)} \leq m_i}^{m_i} {m_i \choose l_1, \cdots, l_{\#(\lambda_i)}}\nonumber\\
&\qquad\qquad\qquad  V_{\alpha(A_i)}^{(m_i-(l_1+\cdots +l_{max(\lambda_i)}))}(w)\left(\prod_{j\in \lambda_i}\frac{\tau_{n_j+l_j}(w,Y_j)}{\tau_{n_j}(w,Y_j)}\right)\Bigg\}dw\nonumber\\
=& \frac{\Gamma(n)}{\Gamma(m+n)}\int_{0}^{\infty}  w^{m}  f_{U_n|\bX}(w|\bX) \prod_{i=1}^{q+1} \Bigg \{\sum_{0 \leq l_1+\cdots +l_{\#(\lambda_i)} \leq m_i}^{m_i} {m_i \choose l_1, \cdots, l_{\#(\lambda_i)}}\nonumber\\
&\qquad\qquad\qquad  V_{\alpha(A_i)}^{(m_i-(l_1+\cdots +l_{\#(\lambda_i)}))}(w)\left(\prod_{j\in \lambda_i}\frac{\tau_{n_j+l_j}(w,Y_j)}{\tau_{n_j}(w,Y_j)}\right)\Bigg\}dw\,.\nonumber
\end{align}
This is part (ii) of the theorem. Then the proof of \cref{them. moment}   is completed.  

\subsection*{Proof of   Theorem \ref{consistency thm}}
We need  the following lemma  to prove  Theorem \ref{consistency thm}.
\begin{lemma}\label{lemma1}
Under the assumption \ref{condition A}, we have for  any $y\in \mathbb{X}$ and $k \in \mathbb{Z}^+$, 
\begin{align}
\lim_{n\rightarrow \infty}\int_0^{\infty}\tau_k(u,y)f_{U_n|\bX}(u|\bX) du = k-C_k(y)\,.\label{proof2.lemma.1}
\end{align}
 
\end{lemma}
\begin{proof}
%Given any $y\in \mathbb{X}$ and $k \in \mathbb{Z}^+$, the result is trivial once we have
%\begin{align}
%\int_0^{\infty}\frac{u}{u+h_k(y)} f_{U_n|\bX}(u|\bX)du \overset{n}{\sim}1\,\label{proof2.lemma.2}
%\end{align}
%as $n\rightarrow \infty$, where $h_k(y) \geq 0$ is finite.
Let $g_n(u)$ be  a constant multiple of  the density   of $f_{U_n|\bX}(u|\bX)$ given by  \eqref{posterior density of U_n}.   Namely, 
 \begin{align}
 g_n(u)=& u^{n-1}e^{-\psi(u)} \prod_{j=1}^{n(\pi)}\tau_{n_j}(u,Y_j)\nonumber\\
 =&u^{n-1}e^{-a\int_{\mathbb{X}}\int_0^{\infty}(1-e^{-us})\rho(ds|x)H(dx)}\prod_{j=1}^{n(\pi)}\int_0^{\infty}s^{n_j}e^{-us}\rho(ds|Y_j)\\
 f_{U_n|\bX}(u|\bX)=& \frac{g_n(u)}{\int_0^\infty g_n(u) du}\,.
 \end{align}
The derivative of $g_n(u)$ is computed  as follows,
 \begin{align}
g_n'(u)=u^{n-2}e^{-\psi_{\mathbb{X}}(u)}\prod_{j=1}^{n(\pi)}\tau_{n_j}(u,Y_j)\left[ n-1-\left(u\int_{\mathbb{X}}\tau_1(u,y)\alpha(dy)+\sum_{j=1}^{n(\pi)}u \frac{\tau_{n_j+1}(u,Y_j)}{\tau_{n_j}(u,Y_j)}\right)\right]\,.\nonumber
\end{align}
Let $h_n(u)=u\int_{\mathbb{X}}\tau_1(u,y)\alpha(dy)$, then $h_n'(u)=\int_{\mathbb{X}}\left(\tau_1(u,y)-u\tau_2(u,y)\right)\alpha(dy)$. By the assumption \ref{condition A}, $u\frac{\tau_2(u,y)}{\tau_1(u,y)}\leq 1$. This means  $h'_n(u)\geq 0$ and then $h_n(u)$ is nondecreasing in $u$.  Similarly,  from  the assumption \ref{condition A},  it follows that $u \frac{\tau_{n_j+1}(u,Y_j)}{\tau_{n_j}(u,Y_j)}$ is  also  nondecreasing in $u$ for all $n_j$. Thus,  we have  
\[
\tilde g_n(u):= u\int_{\mathbb{X}}\tau_1(u,y)\alpha(dy)+\sum_{j=1}^{n(\pi)}u \frac{\tau_{n_j+1}(u,Y_j)}{\tau_{n_j}(u,Y_j)} 
\]
 is nondecreasing in $u$.  Since $g_n(u)$ is a continuously differentiable function such that $\int_0^\infty g_n(u) du<\infty$, it is then bounded and attain its maximum point at some point    $u_{n,n(\pi)}^2$  satisfying $g_n'(u_{n,n(\pi)}^2)=0$ or 
$\tilde g_n( u_{n,n(\pi)}^2)=n-1$.    Note that $\tilde g_n$ is 
also a continuous function and is then bounded on bounded interval. We claim that   $u_{n,n(\pi)}^2\to \infty$ as    $n\rightarrow \infty$. In fact, 
by assumption \ref{condition A}, $u \frac{\tau_{k+1}(u,y)}{\tau_{k}(u,y)}\le \phi (u)(k-C_k(y))$, $\forall k \in \mathbb{Z}^+$ and $y \in \mathbb{X}$, for some function  $\phi (u)\in (0,1)$ which is nondecreasing in $u$ and $\lim_{u\to \infty}\phi (u)=1$. Assume that $u_{n,n(\pi)}^2 < \infty$ as    $n\rightarrow \infty$. Then,
$\phi(u_{n,n(\pi)}^2)=\alpha<1$, which implies $\sum_{j=1}^{n(\pi)}u \frac{\tau_{n_j+1}(u,Y_j)}{\tau_{n_j}(u,Y_j)} <\alpha\left(n-\sum_{j=1}^{n(\pi)}C_j(Y_j)\right)$. Therefore,
\[
n-1=\tilde g_n(u_{n,n(\pi)}^2)<u_{n,n(\pi)}^2\int_{\mathbb{X}}\tau_1(u_{n,n(\pi)}^2,y)\alpha(dy)+\alpha\left(n-\sum_{j=1}^{n(\pi)}C_j(Y_j)\right)\,,
\]
which implies $$n(1-\alpha)+\alpha\sum_{j=1}^{n(\pi)}C_j(Y_j)-1<u_{n,n(\pi)}^2\int_{\mathbb{X}}\tau_1(u_{n,n(\pi)}^2,y)\alpha(dy)<\infty\,,$$
which is a contradiction.

Denote
\[
\tilde \tau_k(u, y)=u \frac{\tau_{k+1}(u,y)}{\tau_{k}(u,y)} \,.
\] 
And let $u_{n,n(\pi)}$ be the positive square  root of $u_{n,n(\pi)}^2$, thus $ u_{n,n(\pi)} \rightarrow \infty $ as $n\rightarrow \infty$. Then, we have the following inequalities,  
\begin{align}
 k-C_k(y) &\geq  \int_0^{\infty}u \frac{\tau_{k+1}(u,y)}{\tau_{k}(u,y)} f_{U_n|\bX}(u|\bX) du
=\frac{\int_0^{\infty}\tilde \tau_k(u, y) g_n(u)du}{\int_0^{\infty}g_n(u)du}\nonumber\\
&\geq \frac{\int_{u_{n,n(\pi)}}^{\infty}\tilde \tau_k(u,y)g_n(u)du}{\int_0^{\infty}g_n(u)du} \geq \tilde \tau_k(u_{n, n(\pi)},  y)\frac{\int_{u_{n,n(\pi)}}^{\infty}g_n(u)du}{\int_0^{\infty}g_n(u)du}\nonumber\\
&=\tilde \tau_k(u_{n, n(\pi)}, y)\left( 1+\frac{\int_0^{u_{n,n(\pi)}}g_n(u)du}{\int_{u_{n,n(\pi)}}^{\infty}g_n(u)du}\right	)^{-1}\nonumber\\
&\geq \tilde \tau_k(u_{n, n(\pi)}, y)\left(1+\frac{\int_0^{u_{n,n(\pi)}}g_n(u)du}{\int_{u_{n,n(\pi)}}^{u_{n,n(\pi)}^2}g_n(u)du}\right)^{-1}\nonumber\\
&\geq \tilde \tau_k(u_{n, n(\pi)}, y)\left(1+\frac{\int_0^{u_{n,n(\pi)}}g_n(u_{n,n(\pi)})du}{\int_{u_{n,n(\pi)}}^{u_{n,n(\pi)}^2}g_n(u_{n,n(\pi)})du}\right)^{-1}\nonumber\\
&=\tilde \tau_k(u_{n, n(\pi)}, y)\left( 1+\frac{u_{n,n(\pi)}g_n(u_{n,n(\pi)})}{(u_{n,n(\pi)}^2-u_{n,n(\pi)})g_n(u_{n,n(\pi)})}\right)^{-1}\nonumber\\
&=(u_{n,n(\pi)}-1) \frac{\tau_{k+1}(u_{n,n(\pi)},y)}{\tau_{k}(u_{n,n(\pi)},y)} \overset{n\rightarrow \infty}{\rightarrow} k-C_k(y)\,.\label{eq: ratio of tau}
\end{align}
The last limit in \cref{eq: ratio of tau} is due to the following form 
\begin{align*}
\lim_{n\rightarrow \infty} (u_{n,n(\pi)}-1) \frac{\tau_{k+1}(u_{n,n(\pi)},y)}{\tau_{k}(u_{n,n(\pi)},y)}=\lim_{n\rightarrow \infty} \frac{(u_{n,n(\pi)}-1)}{u_{n,n(\pi)}} \tilde \tau_k(u_{n, n(\pi)}, y)\,,
\end{align*}
and $\lim_{n\rightarrow \infty}\frac{(u_{n,n(\pi)}-1)}{u_{n,n(\pi)}}=1$, $\lim_{n\rightarrow \infty} \tilde \tau_k(u_{n, n(\pi)}, y) =\lim_{u\rightarrow \infty} \tilde \tau_k(u, y)=k-C_k(y)$ by \cref{condition A}.
This completes the proof of the lemma. 
\end{proof}

Now we are ready to give the  proof of Theorem \ref{consistency thm}. To emphasise the finiteness of $\alpha$, we use the  notation that $\alpha=aH$, where $a=\alpha(\mathbb{X})$ is finite and $H$ is some probability measure.
 
We would   follow the similar     idea  as  that in \citep{freedman1983inconsistent} to  define a class of semi-norms on $\mathbb{M}_{\mathbb{X}}$ such that convergence under such norms implies weak convergence. Let $\mathcal{A}=\{A_i\}_{i=1}^{\infty}$ be a measurable partition of $\mathbb{X}$. The   semi-norm between two probability measures $P_1$ and $P_2$ in $\mathbb{M}_{\mathbb{X}}$ with respect to the partition $\mathcal{A}$ is defined by 
\begin{align}
|P_1-P_2|_{\mathcal{A}}=\sqrt{\sum_{i=1}^{\infty}[P_1(A_i)-P_2(A_i)]^2}\,. \label{semi-norm}
\end{align}
In order to show the posterior distribution of NRMI concentrates around its posterior mean, we have the following lemma.
\begin{lemma}\label{lemma: variance}
For any given measurable partition $\mathcal{A}$, 
\begin{align}
\E\left[ |P-\E[P| \bX]|_{\mathcal{A}}^2| \bX\right]=\sum_{i=1}^{\infty} \Var [P(A_i)|\bX] \rightarrow 0\,,\label{proof2.1}
\end{align}
a.s.-$P_0^{\infty}$ as $n\rightarrow \infty$. 
\end{lemma}
\begin{proof}
To prove this  claim, we shall evaluate  the first and second posterior moments of $P$ for any $A \in \mathcal{X}$.  For the first moment we have 
\begin{align}
&\E[\tP (A)| \bX]=\frac{1}{n}\int_0^{\infty} u f_{U_n}(u) V_{\alpha(A)}^{(1)}(u)du +\frac{1}{n}\sum_{j=1}^{n(\pi)}\int_0^{\infty} u f_{U_n|\bX}(u|\bX) \frac{\tau_{n_j+1}(u,Y_j)}{\tau_{n_j}(u,Y_j)}\delta_{Y_j}(A)du\nonumber\\
&=\frac{{a}}{n}\int_0^{\infty} u f_{U_n}(u) \int_A \tau_1(u,x)H(dx) du +\frac{1}{n}\sum_{j=1}^{n(\pi)}\int_0^{\infty}u f_{U_n|\bX}(u|\bX) \frac{\tau_{n_j+1}(u,Y_j)}{\tau_{n_j}(u,Y_j)}\delta_{Y_j}(A)du\,.\nonumber
\end{align}
For the second moment we have  
\begin{align}
 \E[\tP (A)^2| \bX]=&\frac{a}{n(n+1)}\int_0^{\infty} u^2 f_{U_n|\bX}(u|\bX) \int_A \tau_2(u,x)H(dx) du  \nonumber\\
&+\frac{a^2}{n(n+1)}\int_0^{\infty} u^2 f_{U_n|\bX}(u|\bX) \left(\int_A \tau_1(u,x)H(dx)\right)^2 du \nonumber\\
&+2\frac{a}{n(n+1)}\sum_{j=1}^{n(\pi)}\int_0^{\infty} u^2 f_{U_n|\bX}(u|\bX) \frac{\tau_{n_j+1}(u,Y_j)}{\tau_{n_j}(u,Y_j)}\delta_{Y_j}(A)\int_A \tau_1(u,x)H(dx)du \nonumber\\
&+\frac{1}{n(n+1)}\sum_{j=1}^{n(\pi)}\int_0^{\infty} u^2 f_{U_n|\bX}(u|\bX) \frac{\tau_{n_j+2}(u,Y_j)}{\tau_{n_j}(u,Y_j)}\delta_{Y_j}(A)du \nonumber\\
&+2\frac{1}{n(n+1)}\sum_{j\neq k}^{n(\pi)}\int_0^{\infty} u^2 f_{U_n|\bX}(u|\bX) \frac{\tau_{n_k+1}(u,Y_k)}{\tau_{n_k}(u,Y_k)}\frac{\tau_{n_j+1}(u,Y_j)}{\tau_{n_j}(u,Y_j)}\delta_{Y_i}(A)\delta_{Y_j}(A)du\,.\nonumber
\end{align}
Then, we can write 
\[
\sum_{i=1}^{\infty} \Var [P(A_i)|\bX]=\sum_{i=1}^{\infty}\left(\E [\tP (A)^2| \bX]-\E[\tP (A)| \bX]^2\right)=\mathcal{J}_1+\mathcal{J}_2+\mathcal{J}_3+\mathcal{J}_4\,,
\]
 where the terms $\mathcal{J}_1$, $\mathcal{J}_2$, $\mathcal{J}_3$, $\mathcal{J}_4$ are defined as follows.
\begin{align}
 \mathcal{J}_1=&\frac{a}{n(n+1)}\int_0^{\infty} u^2 f_{U_n|\bX}(u|\bX) \int_{\mathbb{X}} \tau_2(u,x)H(dx) du \nonumber\\
&+\frac{a^2}{n(n+1)}\sum_{i=1}^{\infty}\int_0^{\infty} u^2 f_{U_n|\bX}(u|\bX) \left(\int_{A_i} \tau_1(u,x)H(dx)\right)^2 du \nonumber\\
&-\frac{a^2}{n^2}\sum_{i=1}^{\infty}\left(\int_0^{\infty} u f_{U_n}(u) \int_{A_i} \tau_1(u,x)H(dx) du\right)^2\,;\label{proof2.J1}
\end{align}
\begin{align}
 \mathcal{J}_2=&2\frac{a}{n(n+1)}\sum_{i=1}^{\infty}\sum_{j=1}^{n(\pi)}\int_0^{\infty} u^2 f_{U_n|\bX}(u|\bX) \frac{\tau_{n_j+1}(u,Y_j)}{\tau_{n_j}(u,Y_j)}  \delta_{Y_j}(A_i)\nonumber\\
&\qquad\qquad \times\int_{A_i} \tau_1(u,x)H(dx)du  -2\frac{a}{n^2}\sum_{i=1}^{\infty}\sum_{j=1}^{n(\pi)}\int_0^{\infty} u f_{U_n}(u) \int_{A_i} \tau_1(u,x)H(dx) du\nonumber\\
&\qquad\qquad\times \int_0^{\infty} u f_{U_n|\bX}(u|\bX) \frac{\tau_{n_j+1}(u,Y_j)}{\tau_{n_j}(u,Y_j)}\delta_{Y_j}(A_i)du\,;\label{proof2.J2}
\end{align}
\begin{align}
 \mathcal{J}_3=&\frac{1}{n(n+1)}\sum_{i=1}^{\infty}\sum_{j=1}^{n(\pi)}\int_0^{\infty} u^2 f_{U_n|\bX}(u|\bX) \frac{\tau_{n_j+2}(u,Y_j)}{\tau_{n_j}(u,Y_j)}\delta_{Y_j}(A_i)du \nonumber\\
&\qquad\qquad-\frac{1}{n^2}\sum_{i=1}^{\infty}\sum_{j=1}^{n(\pi)}\left(\int_0^{\infty} uf_{U_n|\bX}(u|\bX) \frac{\tau_{n_j+1}(u,Y_j)}{\tau_{n_j}(u,Y_j)}\delta_{Y_j}(A_i)du\right)^2\\
&=\frac{1}{n(n+1)}\sum_{j=1}^{n(\pi)}\int_0^{\infty} u^2 f_{U_n|\bX}(u|\bX) \frac{\tau_{n_j+2}(u,Y_j)}{\tau_{n_j}(u,Y_j)}du \nonumber\\
&\qquad\qquad-\frac{1}{n^2}\sum_{j=1}^{n(\pi)}\left(\int_0^{\infty} uf_{U_n|\bX}(u|\bX) \frac{\tau_{n_j+1}(u,Y_j)}{\tau_{n_j}(u,Y_j)}du\right)^2\,; \label{proof2.J3}
\end{align}
and 
\begin{align}
 \mathcal{J}_4=&2\frac{1}{n(n+1)}\sum_{i=1}^{\infty}\sum_{j\neq k}^{n(\pi)}\int_0^{\infty} u^2 f_{U_n|\bX}(u|\bX) \frac{\tau_{n_k+1}(u,Y_k)}{\tau_{n_k}(u,Y_k)}\frac{\tau_{n_j+1}(u,Y_j)}{\tau_{n_j}(u,Y_j)}\delta_{Y_k}(A_i)\delta_{Y_j}(A_i)du \nonumber\\
&\qquad\qquad-2\frac{1}{n^2}\sum_{i=1}^{\infty}\sum_{j\neq k}^{n(\pi)}\int_0^{\infty} u f_{U_n|\bX}(u|\bX) \frac{\tau_{n_k+1}(u,Y_k)}{\tau_{n_k}(u,Y_k)}\delta_{Y_k}(A_i)du \nonumber\\
& \qquad\qquad \times\int_0^{\infty} u f_{U_n|\bX}(u|\bX)\frac{\tau_{n_j+1}(u,Y_j)}{\tau_{n_j}(u,Y_j)}\delta_{Y_j}(A_i)du\,.\label{proof2.J4}
\end{align}
We will first consider the terms $\mathcal{J}_2$, $\mathcal{J}_3$, $\mathcal{J}_4$   and    then $\mathcal{J}_1$. But before dealing with them, we need  some prior preparations.    By the identity $\E[\tP (\mathbb{X})| \bX]=1$ we have 
\begin{align}
&\frac{a}{n}\int_0^{\infty} u f_{U_n}(u) \int_{\mathbb{X}} \tau_1(u,x)H(dx) du +\frac{1}{n}\sum_{j=1}^{n(\pi)}\int_0^{\infty}u f_{U_n|\bX}(u|\bX) \frac{\tau_{n_j+1}(u,Y_j)}{\tau_{n_j}(u,Y_j)}du=1\,.\nonumber
\end{align}
By Lemma \ref{lemma1}, we have the approximation 
\begin{align}
\frac{a}{n}\int_0^{\infty} u f_{U_n|\bX}(u) \int_{\mathbb{X}} \tau_1(u,x)H(dx) du \overset{n}{\sim} \frac{\sum_{j=1}^{n(\pi)} C_{n_j}(Y_j)}{n}\label{appro.1}
\end{align}
 as $n$ is large.
On the other hand,  let $u_{n,n(\pi)}$ be the maximal point of $g_n(u)$ as in Lemma \ref{lemma1}. Under the assumption \ref{condition A}, we know that $u\tau_1(u,x)$ is nondecreasing in $u$ for all $x$.  We have 
\begin{align}
&a\int_0^{u_{n,n(\pi)}} u f_{U_n|\bX}(u) \int_{\mathbb{X}} \tau_1(u,x)H(dx) du\\
&\qquad\qquad = a\frac{\int_0^{u_{n,n(\pi)}} u g_n(u) \int_{\mathbb{X}} \tau_1(u,x)H(dx) du}{\int_0^{\infty}g_n(u)du}\nonumber\\
&\qquad\qquad \leq au_{n,n(\pi)}\int_{\mathbb{X}} \tau_1(u_{n,n(\pi)},x)H(dx) \frac{\int_0^{u_{n,n(\pi)}} g_n(u)du}{\int_0^{\infty}g_n(u)du}\nonumber\\
&\qquad\qquad =au_{n,n(\pi)}\int_{\mathbb{X}} \tau_1(u_{n,n(\pi)},x)H(dx)\left(1+\frac{\int_{u_{n,n(\pi)}}^{\infty}g_n(u)du}{\int_0^{u_{n,n(\pi)}} g_n(u)du} \right)^{-1}\nonumber\\
&\qquad\qquad \leq au_{n,n(\pi)}\int_{\mathbb{X}} \tau_1(u_{n,n(\pi)},x)H(dx)\left(1+\frac{\int_{u_{n,n(\pi)}}^{u_{n,n(\pi)}^2}g_n(u_{n,n(\pi)})du}{\int_0^{u_{n,n(\pi)}} g_n(u_{n,n(\pi)})du} \right)^{-1}\nonumber\\
&\qquad\qquad = au_{n,n(\pi)}\int_{\mathbb{X}} \tau_1(u_{n,n(\pi)},x)H(dx)\left(1+\frac{u_{n,n(\pi)}(u_{n,n(\pi)} - 1)g_n(u_{n,n(\pi)})}{u_{n,n(\pi)}g_n(u_{n,n(\pi)})} \right)^{-1}\nonumber\\
&\qquad\qquad =a\int_{\mathbb{X}} \tau_1(u_{n,n(\pi)},x)H(dx)\,,
\end{align}
which goes to $0$ as $n \rightarrow \infty$ by the Monotone convergence theorem, since $\tau_1(u,x)$ is decreasing to $0$ in $u$ for all $x$. Combining  the above computation with the approximation \eqref{appro.1}, we have 
\begin{align}
\lim_{n\rightarrow \infty}\frac{a}{\sum_{j=1}^{n(\pi)} C_{n_j}(Y_j)}\int_{u_{n,n(\pi)}}^{\infty} u f_{U_n|\bX}(u) \int_{\mathbb{X}} \tau_1(u,x)H(dx) du=1 \,.
\label{appro.1.1}
\end{align}
\textbf{Step 1: Evaluation of $\mathcal{J}_2$.}

Notice first that for any $A_i$ and $Y_j$, by the assumption \ref{condition A}, we will have
\begin{align}
I_1:=&\int_0^{\infty} u^2 \frac{\tau_{n_j+1}(u,Y_j)}{\tau_{n_j}(u,Y_j)}\delta_{Y_j}(A_i)\int_{A_i} \tau_1(u,x)H(dx)f_{U_n|\bX}(u|\bX) du\nonumber \\
&\leq \left(n_j-C_{n_j}(Y_j)\right)\delta_{Y_j}(A_i)\int_0^{\infty}  uf_{U_n|\bX}(u|\bX) \int_{A_i} \tau_1(u,x)H(dx)du\,.\label{appro.1.2}
\end{align}
On the other hand, 
\begin{align}
%&\int_0^{\infty} u^2 \frac{\tau_{n_j+1}(u,Y_j)}{\tau_{n_j}(u,Y_j)}\delta_{Y_j}(A_i)\int_{A_i} \tau_1(u,x)H(dx)f_{U_n|\bX}(u|\bX) du \nonumber\\
I_1&\geq \int_{u_{n,n(\pi)}}^{\infty} u^2 \frac{\tau_{n_j+1}(u,Y_j)}{\tau_{n_j}(u,Y_j)}\delta_{Y_j}(A_i)\int_{A_i} \tau_1(u,x)H(dx)f_{U_n|\bX}(u|\bX) du\nonumber \\
&\geq u_{n,n(\pi)}\frac{\tau_{n_j+1}(u_{n,n(\pi)},Y_j)}{\tau_{n_j}(u_{n,n(\pi)},Y_j)}\delta_{Y_j}(A_i)\int_{u_{n,n(\pi)}}^{\infty} u f_{U_n|\bX}(u) \int_{\mathbb{X}} \tau_1(u,x)H(dx) du\,.\label{appro.1.3}
\end{align}
By the above inequalities \eqref{appro.1.2}, \eqref{appro.1.3} and the approximation \eqref{appro.1}, \eqref{appro.1.1}, we can see as $n$ becomes large
\begin{align}
&\int_0^{\infty} u^2 f_{U_n|\bX}(u|\bX) \frac{\tau_{n_j+1}(u,Y_j)}{\tau_{n_j}(u,Y_j)}\delta_{Y_j}(A_i)\int_{A_i} \tau_1(u,x)H(dx)du \nonumber\\
& \overset{n}{\sim}\int_0^{\infty} u f_{U_n}(u) \int_{A_i} \tau_1(u,x)H(dx) du\int_0^{\infty} u f_{U_n|\bX}(u|\bX) \frac{\tau_{n_j+1}(u,Y_j)}{\tau_{n_j}(u,Y_j)}\delta_{Y_j}(A_i)du\,.\nonumber
\end{align}
Thus, for large $n$, we have
\begin{align}
&\mathcal{J}_2  \overset{n}{\sim} 2\left(\frac{a}{n(n+1)}-\frac{a}{n^2}\right)\sum_{i=1}^{\infty}\sum_{j=1}^{n(\pi)}\left(n_j-C_{n_j}(Y_j)\right)\delta_{Y_j}(A_i) \nonumber\\
&\quad\qquad\qquad  \times\int_0^{\infty} uf_{U_n|\bX}(u|\bX) \int_{A_i} \tau_1(u,x)H(dx)du\nonumber\\
&  \overset{n}{\sim}  2\left(\frac{a}{n(n+1)}-\frac{a}{n^2}\right)\sum_{j=1}^{n(\pi)}\left(n_j-C_{n_j}(Y_j)\right)\int_0^{\infty} uf_{U_n|\bX}(u|\bX) \int_{\mathbb{X}} \tau_1(u,x)H(dx)du\,.\nonumber
\end{align}
This combined with  \eqref{appro.1} yields 
\begin{align}
&\mathcal{J}_2  \overset{n}{\sim}  -2\frac{\left(n-\sum_{j=1}^{n(\pi)}C_{n_j}(Y_j)\right)\left(\sum_{j=1}^{n(\pi)} C_{n_j}(Y_j)\right)}{n^2(n+1)}\,,
\end{align}
which has order   $O(\frac{1}{n})$.

\textbf{Step 2: Evaluation of $\mathcal{J}_3$.}

For $\mathcal{J}_3$,   notice that under the assumption \ref{condition A}, we have 
\[
u^2\frac{\tau_{n_j+2}(u,Y_j)}{\tau_{n_j}(u,Y_j)}=u\frac{\tau_{n_j+2}(u,Y_j)}{\tau_{n_j+1}(u,Y_j)}\times u\frac{\tau_{n_j+1}(u,Y_j)}{\tau_{n_j}(u,Y_j)}
\]
  is nondecreasing in $u$ and is bounded by $(n_j+1-C_{n_j+1}(Y_j))(n_j-C_{n_j}(Y_j))$.   Using a similar approach as that in Lemma \ref{lemma1}, we have as $n$ is large,
\begin{align}
\int_0^{\infty} u^2 f_{U_n|\bX}(u|\bX) \frac{\tau_{n_j+2}(u,Y_j)}{\tau_{n_j}(u,Y_j)}du \overset{n}{\sim} (n_j+1-C_{n_j+1}(Y_j))(n_j-C_{n_j}(Y_j))\,.
\end{align}
Combining it with   Lemma \ref{lemma1}, we have as $n$ becomes large
\begin{align}
 \mathcal{J}_3 \overset{n}{\sim}& \frac{1}{n(n+1)}\sum_{j=1}^{n(\pi)} (n_j+1-C_{n_j+1}(Y_j))(n_j-C_{n_j}(Y_j)) -\frac{1}{n^2}\sum_{j=1}^{n(\pi)} (n_j-C_{n_j}(Y_j))^2\nonumber\\
  = &\frac{1}{n^2(n+1)}\sum_{j=1}^{n(\pi)} (n_j-C_{n_j}(Y_j))\left(n+(n+1)C_{n_j}-n_j-nC_{n_j+1}\right)\nonumber\\
  \leq & 2\frac{n-\left(\sum_{j=1}^{n(\pi)} C_{n_j}(Y_j)\right)}{n(n+1)}\,,
\end{align}
which has order at most $O(\frac{1}{n})$.

\textbf{Step 3: Evaluation of $\mathcal{J}_4$.}

For $\mathcal{J}_4$, we have  that  under the assumption \ref{condition A},   $u^2\frac{\tau_{n_k+1}(u,Y_k)}{\tau_{n_k}(u,Y_k)}\frac{\tau_{n_j+1}(u,Y_j)}{\tau_{n_j}(u,Y_j)}$ is nondecreasing in $u$ and is bounded by $(n_k-C_{n_k}(Y_k))(n_j-C_{n_j}(Y_j))$. Using a  similar argument to that  in Lemma \ref{lemma1}   leads  to
\begin{align}
&\int_0^{\infty} u^2 f_{U_n|\bX}(u|\bX) \frac{\tau_{n_k+1}(u,Y_k)}{\tau_{n_k}(u,Y_k)}\frac{\tau_{n_j+1}(u,Y_j)}{\tau_{n_j}(u,Y_j)}du \nonumber\\
&\qquad\qquad \overset{n}{\sim}  \int_0^{\infty} u f_{U_n|\bX}(u|\bX) \frac{\tau_{n_k+1}(u,Y_k)}{\tau_{n_k}(u,Y_k)}du\int_0^{\infty} u f_{U_n|\bX}(u|\bX)\frac{\tau_{n_j+1}(u,Y_j)}{\tau_{n_j}(u,Y_j)}du \nonumber\\
&\qquad\qquad\overset{n}{\sim}  (n_k-C_{n_k}(Y_k))(n_j-C_{n_j}(Y_j))\,.\nonumber
\end{align}
Thus
\begin{align}
&\mathcal{J}_4\overset{n}{\sim} 2\left(\frac{1}{n(n+1)}-\frac{1}{n^2}\right)\sum_{i=1}^{\infty}\sum_{j\neq k}^{n(\pi)}(n_k-C_{n_k}(Y_k))(n_j-C_{n_j}(Y_j))\delta_{Y_k}(A_i)\delta_{Y_j}(A_i)\nonumber\\
&\quad \overset{n}{\sim} -\frac{2\sum_{j\neq k}^{n(\pi)}(n_k-C_{n_k}(Y_k))(n_j-C_{n_j}(Y_j))}{n^2(n+1)}\,,\nonumber
\end{align}
which has an order at most $O(\frac{1}{n})$.

\textbf{Step 4: Evaluation of $\mathcal{J}_1$.}

Finally, we deal with  the term $\mathcal{J}_1$. Notice that $\E[\tP (\mathbb{X})^2| \bX]=1$.
Using the computation we obtained for  $\mathcal{J}_2$, $\mathcal{J}_3$, $\mathcal{J}_4$, we have 
\begin{align}
 1= \E[\tP (\mathbb{X})^2| \bX]=&\frac{a}{n(n+1)}\int_0^{\infty} u^2 f_{U_n|\bX}(u|\bX) \int_{\mathbb{X}} \tau_2(u,x)H(dx) du  \nonumber\\
&  +\frac{a^2}{n(n+1)}\int_0^{\infty} u^2 f_{U_n|\bX}(u|\bX) \left(\int_{\mathbb{X}} \tau_1(u,x)H(dx)\right)^2 du \nonumber\\
&+2\frac{a}{n(n+1)}\sum_{j=1}^{n(\pi)}\int_0^{\infty} u^2 f_{U_n|\bX}(u|\bX) \frac{\tau_{n_j+1}(u,Y_j)}{\tau_{n_j}(u,Y_j)}\int_{\mathbb{X}} \tau_1(u,x)H(dx)du \nonumber\\
&+\frac{1}{n(n+1)}\sum_{j=1}^{n(\pi)}\int_0^{\infty} u^2 f_{U_n|\bX}(u|\bX) \frac{\tau_{n_j+2}(u,Y_j)}{\tau_{n_j}(u,Y_j)}du \nonumber\\
&+2\frac{1}{n(n+1)}\sum_{j\neq k}^{n(\pi)}\int_0^{\infty} u^2 f_{U_n|\bX}(u|\bX) \frac{\tau_{n_k+1}(u,Y_k)}{\tau_{n_k}(u,Y_k)}\frac{\tau_{n_j+1}(u,Y_j)}{\tau_{n_j}(u,Y_j)}du\nonumber\\
 \overset{n}{\sim} &\frac{a}{n(n+1)}\int_0^{\infty} u^2 f_{U_n|\bX}(u|\bX) \int_{\mathbb{X}} \tau_2(u,x)H(dx) du  \nonumber\\
&+\frac{a^2}{n(n+1)}\int_0^{\infty} u^2 f_{U_n|\bX}(u|\bX) \left(\int_{\mathbb{X}} \tau_1(u,x)H(dx)\right)^2 du\nonumber\\
&+2\frac{\left(\sum_{j=1}^{n(\pi)} C_{n_j}(Y_j)\right)\left(n-\left(\sum_{j=1}^{n(\pi)} C_{n_j}(Y_j)\right)\right)}{n(n+1)}\nonumber\\
&+\frac{\sum_{j=1}^{n(\pi)}(n_j+1-C_{n_j+1}(Y_j)-\frac{n_j-C_{n_j}}{n})(n_j-C_{n_j}(Y_j))}{n(n+1)}\nonumber\\
&+2\frac{\sum_{j\neq k}^{n(\pi)}(n_k-C_{n_k}(Y_k))(n_j-C_{n_j}(Y_j))}{n(n+1)}\,.  \nonumber
\end{align}
This  implies 
\begin{align}
&\frac{a}{n(n+1)}\int_0^{\infty} u^2 f_{U_n|\bX}(u|\bX) \int_{\mathbb{X}} \tau_2(u,x)H(dx) du  \nonumber\\
&+\frac{a^2}{n(n+1)}\int_0^{\infty} u^2 f_{U_n|\bX}(u|\bX) \left(\int_{\mathbb{X}} \tau_1(u,x)H(dx)\right)^2 du   \nonumber\\
& \overset{n}{\sim} \frac{n+\left(\sum_{j=1}^{n(\pi)} C_{n_j}(Y_j)\right)^2-\sum_{j=1}^{n(\pi)}(n_j-C_{n_j}(Y_j))(1+C_{n_j}-C_{n_j+1}-\frac{n_j-C_{n_j}}{n})}{n(n+1)}\,.\label{appro.2}
\end{align}
Combining  the approximations \eqref{appro.1} and \eqref{appro.2}, we have 
\begin{align}
 \mathcal{J}_1=&\frac{a}{n(n+1)}\int_0^{\infty} u^2 f_{U_n|\bX}(u|\bX) \int_{\mathbb{X}} \tau_2(u,x)H(dx) du\nonumber\\
&+\frac{a^2}{n(n+1)}\int_0^{\infty} u^2 f_{U_n|\bX}(u|\bX) \left(\int_{\mathbb{X}} \tau_1(u,x)H(dx)\right)^2 du\nonumber\\
&-\frac{a^2}{n^2}\left(\int_0^{\infty} u f_{U_n}(u) \int_{\mathbb{X}} \tau_1(u,x)H(dx) du\right)^2\nonumber\\
&+2\sum_{i\neq l}^{\infty}\int_0^{\infty} u f_{U_n}(u) \int_{A_i} \tau_1(u,x)H(dx) du\int_0^{\infty} u f_{U_n}(u) \int_{A_l} \tau_1(u,x)H(dx) du\nonumber\\
&-2\sum_{i\neq l}^{\infty}\int_0^{\infty} u^2 f_{U_n|\bX}(u|\bX)\int_{A_i} \tau_1(u,x)H(dx)\int_{A_l} \tau_1(u,x)H(dx) du\nonumber\\
 \overset{n}{\sim}& \frac{n+\left(\sum_{j=1}^{n(\pi)} C_{n_j}(Y_j)\right)^2-\sum_{j=1}^{n(\pi)}(n_j-C_{n_j}(Y_j))(1+C_{n_j}-C_{n_j+1}-\frac{n_j-C_{n_j}}{n})}{n(n+1)}\nonumber\\
& -\frac{\left(\sum_{j=1}^{n(\pi)} C_{n_j}(Y_j)\right)^2}{n^2}\nonumber\\
&+2\frac{a^2}{n^2}\sum_{i\neq l}^{\infty}\int_0^{\infty} u f_{U_n}(u) \int_{A_i} \tau_1(u,x)H(dx) du\int_0^{\infty} u f_{U_n}(u) \int_{A_l} \tau_1(u,x)H(dx) du\nonumber\\
&-2\frac{a^2}{n(n+1)}\sum_{i\neq l}^{\infty}\int_0^{\infty} u^2 f_{U_n|\bX}(u|\bX)\int_{A_i} \tau_1(u,x)H(dx)\int_{A_l} \tau_1(u,x)H(dx) du\,.\label{proof2.3}
\end{align}
We now treat  the above last two  summation terms. First, we have  
\begin{align}
&2\sum_{i\neq l}^{\infty}\int_0^{\infty} u f_{U_n}(u) \int_{A_i} \tau_1(u,x)H(dx) du\int_0^{\infty} u f_{U_n}(u) \int_{A_l} \tau_1(u,x)H(dx) du\nonumber\\
&\overset{n}{\sim}\left(\int_0^{\infty} u f_{U_n}(u) \int_{\mathbb{X}} \tau_1(u,x)H(dx) du\right)^2\nonumber
\end{align}
and
\begin{align}
&2\sum_{i\neq l}^{\infty}\int_0^{\infty} u^2 f_{U_n|\bX}(u|\bX)\int_{A_i} \tau_1(u,x)H(dx)\int_{A_l} \tau_1(u,x)H(dx) du\nonumber\\
&\overset{n}{\sim} \int_0^{\infty} u^2 f_{U_n|\bX}(u|\bX) \left(\int_{\mathbb{X}} \tau_1(u,x)H(dx)\right)^2 du\,.\nonumber
\end{align}
%which means they have the same order that no larger than $\max\{n+\left(\sum_{j=1}^{n(\pi)} C_{n_j}(Y_j)\right)^2-\sum_{j=1}^{n(\pi)}(n_j-C_{n_j}(Y_j))(1+C_{n_j}-C_{n_j+1}), \left(\sum_{j=1}^{n(\pi)} C_{n_j}(Y_j)\right)^2\}$.
Thus,    
\begin{align}
&\mathcal{J}_1 \overset{n}{\sim} \frac{n^2-n\sum_{j=1}^{n(\pi)}(n_j-C_{n_j}(Y_j))(1+C_{n_j}-C_{n_j+1})-\left(\sum_{j=1}^{n(\pi)} C_{n_j}(Y_j)\right)^2}{n^2(n+1)}\,.\nonumber
\end{align}
It is easy to have that 
\[
\sum_{j=1}^{n(\pi)}(n_j-C_{n_j}(Y_j))(1+C_{n_j}-C_{n_j+1}) \leq 3n\sum_{j=1}^{n(\pi)}(n_j-C_{n_j}(Y_j)) \leq 3n^2\,.
\]
Thus $\mathcal{J}_1$ has an order $O(\frac{1}{n})$.

Summarizing the above four steps for evaluating $\mathcal{J}_1$, $\mathcal{J}_2$, $\mathcal{J}_3$, $\mathcal{J}_4$, we have  \[
\sum_{i=1}^{\infty} \Var [P(A_i)|\bX] \overset{n}{\sim} O(\frac{1}{n}) \rightarrow 0
\qquad \hbox{ as $n \rightarrow \infty$}\,.
\]

\end{proof}

Now, we can give the completion of \cref{consistency thm}.
\begin{proof}

By \cref{lemma: variance}, the distribution of $P(\cdot |\bX)$ converges weakly to the point mass at the distribution of $\lim_{n \rightarrow \infty}\E[P(dx)|\bX]$.

If the ``true" distribution $P_0$ of $\bX$ is continuous, the posterior expectation has the following form for any $A \in \mathcal{X}$.
\begin{align}
 \E[P(A)|\bX]=&\frac{a}{n}\int_0^{\infty} u f_{U_n}(u) \int_A \tau_1(u,x)H(dx) du \nonumber\\
&\quad+\frac{1}{n}\sum_{j=1}^{n}\int_0^{\infty}u f_{U_n|\bX}(u|\bX) \frac{\tau_{2}(u,X_j)}{\tau_{1}(u,X_j)}\delta_{X_j}(A)du\,.\label{proof2.4}
\end{align}
As $n \rightarrow \infty$, by \eqref{eq: ratio of tau}, the weight   $\int_0^{\infty}u f_{U_n|\bX}(u|\bX) \frac{\tau_{2}(u,X_j)}{\tau_{1}(u,X_j)}du \overset{n}{\sim} 1-C_1(X_j)$, thus the second part of the summation in \eqref{proof2.4} has the form $\sum_{j=1}^n \frac{1-C_1(X_j)}{n} \delta_{X_j}(A)$ that converges uniformly over Glivenko-Cantelli classes to $(1-\bar{C_1})P_0$. Since the sum of the weights of $H(dx)$ and $\delta_{X_j}(dx)$ is equal to $1$, we have $\lim_{n \rightarrow \infty}\E[P(\cdot)|\bX]=\bar{C_1}H(\cdot)+(1-\bar{C_1})P_0(\cdot)$.

If the ``true" distribution $P_0$ of $\bX$ is discrete with $\lim_{n\rightarrow \infty}\frac{n(\pi)}{n}=0$ that is true a.s., the posterior expectation has the following form:
% for any $A \in \mathcal{X}$.
\begin{align}
&\E[P(A)|\bX]=\frac{a}{n}\int_0^{\infty} u f_{U_n}(u) \int_A \tau_1(u,x)H(dx) du\nonumber\\
&\quad +\frac{1}{n}\sum_{j=1}^{n(\pi)}\int_0^{\infty}u f_{U_n|\bX}(u|\bX) \frac{\tau_{n_j+1}(u,Y_j)}{\tau_{n_j}(u,Y_j)}\delta_{Y_j}(A)du\,.\label{proof2.5}
\end{align}
As $n \rightarrow \infty$, by \eqref{eq: ratio of tau},  $\int_0^{\infty}u f_{U_n|\bX}(u|\bX) \frac{\tau_{n_j+1}(u,Y_j)}{\tau_{n_j}(u,Y_j)} \overset{n}{\sim} n_j-C_{n_j}(Y_j)$. Hence,  the second part of the summation in \eqref{proof2.5} has the form $\sum_{j=1}^{n(\pi)} \frac{n_j-C_{n_j}(Y_j)}{n} \delta_{Y_j}(A)$. Notice that 
\begin{align*}
\sum_{j=1}^{n(\pi)} \frac{n_j-C_{n_j}(Y_j)}{n} \delta_{Y_j}(A)=\sum_{i=1}^n \frac{1}{n} \delta_{X_i}(A)-\sum_{j=1}^{n(\pi)} \frac{C_{n_j}(Y_j)}{n} \delta_{Y_j}(A)\,,
\end{align*}
where the term $\sum_{j=1}^{n(\pi)} \frac{C_{n_j}(Y_j)}{n} \leq \sum_{j=1}^{n(\pi)} \frac{1}{n}=\frac{n(\pi)}{n}$ converges to $0$. Thus the weight of $H(dx)$ is $0$ and we have $\lim_{n \rightarrow\infty }\E[P(\cdot)|\bX]=P_0(\cdot)$.
This completes the proof of Theorem 
\ref{consistency thm}. 
\end{proof}
\subsection*{Proof of \cref{thm: BVM}}\label{proof: BVM}

As the preparation of the proof of \cref{thm: BVM}, we shall present the posterior process of $\NGGP(a, \sigma, \theta, H)$  followed by \cref{posterior}.
\begin{lemma}\label{lemma: NGGP posterior}
If $P\sim \NGGP(a, \sigma, \theta, H)$,  conditionally on $\bX$ and a latent random variable $U_n$, $P$ coincides in distribution with the random probability measure
\begin{align}
\kappa_n P_{U_n} + (1-\kappa_n) \sum_{j=1}^{n(\pi)} D_{n,j}\delta_{Y_j}\,,\label{NGGP posterior distribution}
\end{align}
where 
\begin{enumerate}
\item[(i)] The random variable $U_n$ has density
\begin{align}
f_{U_n}(u) \varpropto \frac{u^{n-1}}{(u+\theta)^{n-n(\pi)\sigma}}e^{-\frac{a}{\sigma}(u+\theta)^{\sigma}}\,.  \label{NGGP density of U_n}
\end{align}
\item[(ii)] Given $U_n=u$,  $P_{U_n}\sim \NGGP(a, \sigma, \theta+u, H)$.
\item[(iii)]  $D_n:=(D_{n,1},\cdots, D_{n,n(\pi)}) \sim \Dir (n(\pi);n_1-\sigma, \cdots\, n_{n(\pi)}-\sigma)$ is independent of $\kappa_n$ and $P_{U_n}$.
\item[(iv)] The random elements $P_{U_n}$ and $J_j$, $j\in \{1, \cdots, n(\pi)\}$ are independent.   
\item[(v)]   $T_{(U_n)}=\tilde{\mu}_{(U_n)}(\mathbb{X})$
and  $\kappa_n=\frac{T_{(U_n)}}{T_{(U_n)}+\sum_{j=1}^{n(\pi)}J_j}$. 
%\item[(vi)]   The conditional density of $U_n$ given $\textbf{X}$ is given by
%\begin{align}
%f_{U_n|\textbf{X}}(u|\textbf{X}) \propto u^{n-1}e^{-\psi(u)} \prod_{j=1}^{n(\pi)}\tau_{n_j}(u,Y_j)\,.\label{posterior density of U_n}
%\end{align} 
\end{enumerate}
\end{lemma} 
\begin{proof}
The lemma is an immediate consequence of  \cref{posterior} and the NGGP intensity given in \cref{example: NGGP}  except (iii), where we have a more specific form 
 for $D_n$. To verify (iii), we let  $D_{n,j}:=\frac{J_j}{\sum_{j=1}^{n(\pi)} J_j}$. Since $\{ J_1, \cdots, J_{n(\pi)}\}$  are independent $G(n_j-\sigma, u+\theta)$ random variables with density
\begin{align*}
f_{J_j}(t|U_n=u) = \frac{(u+\theta)^{n_j-\sigma}}{\Gamma(n_j-\sigma)}t^{n_j-\sigma-1}e^{-(u+\theta)t}\,,
\end{align*}
By the Proposition G.2 in \citep{ghosal2017fundamentals}, we have $$D_n:=(n(\pi);D_{n,1},\cdots, D_{n,n(\pi)}) \sim \Dir (n_1-\sigma, \cdots\, n_{n(\pi)}-\sigma),$$ which is totally independent of $U_n$, thus independent of $\kappa_n$ and $P_{U_n}$. To understand the independence, we can use the relationship between Dirichlet distribution and the gamma distribution from the Proposition G.2 in \citep{ghosal2017fundamentals}.
% and let $D_{n,j}=\frac{\gamma_j}{\sum_{j=1}^{n(\pi)} \gamma_j}$, where $\gamma_j \sim G(n_j-\sigma, 1)$.
\end{proof}
The convergences \eqref{eq1: BVM dicrete} and \eqref{eq2: BVM dicrete} are equivalent in Theprem \ref{thm: BVM}, and also the convergences \eqref{eq1: BVM continuous} and \eqref{eq2: BVM continuous} are equivalent. These equivalences can be shown by the following lemma. To make the results lavish, we will assume $\{\sigma_i\}_{i=1}^n$ be a sequence such that $\lim_{n \rightarrow \infty} \sigma_n=\sigma \in [0,1)$ in the following proofs. It is worth to point that, we always assume that $\sigma_i<1$ and $\sigma<1$ to make sure all quantities in this work are well-defined. To be more precise, this assumption would make the forms $\int_0^{\infty} s^{n_j-\sigma_i-1}e^{-(u+\theta)s}ds <\infty$ and  $\int_0^{\infty} s^{n_j-\sigma-1}e^{-(u+\theta)s}ds <\infty$ for any integer $n_j\geq 1 $. 
\begin{lemma}\label{lemma: equivalent}
For any $P_0$, we have 
\begin{itemize}
\item[(i)] If $P_0$ is discrete, 
\begin{align}
\lim_{n\rightarrow \infty} \E[P|\bX]=\lim_{n\rightarrow \infty}\mathbb{P}_n+ \frac{\sigma_n n(\pi)}{n}(H-\tilde{\mathbb{P}}_n)=P_0\,.\label{eq: discrete case mean}
\end{align}
\item[(ii)] If $P_0$ is continuous, 
\begin{align}
\lim_{n\rightarrow \infty} \E[P|\bX]=\lim_{n\rightarrow \infty} (1-\sigma_n)\mathbb{P}_n+\sigma_n H = (1-\sigma_n)P_0+\sigma_n H\,.\label{eq: continuous case mean}
\end{align}
\end{itemize}
\end{lemma}
\begin{proof}
Since the convergence of $\sigma_n$ to $\sigma$ won't affect the proof, and $\sigma_n$ is well-defined as discussed previously, we may fix $\sigma_n$ and use $\sigma$ for the sake of notational simplicity in the proof.

By applying the NGGP intensity \eqref{NGGP intensity} to \cref{them. moment}, we have 
\begin{align}
\E[P|\bX]=\frac{a}{n}\int_0^{\infty} \frac{u}{(u+\theta)^{1-\sigma}}f_{U_n}(u)du H+\frac{1}{n}\sum_{j=1}^{n(\pi)}\int_0^{\infty} \frac{(n_j-\sigma)u}{u+\theta}f_{U_n}(u)du \delta_{Y_j}\,.\label{eq.mean of NGGP}
\end{align}
To evaluate $\lim_{n\rightarrow \infty}\E[P|\bX] $, we need to find the limits of 
\begin{align}
&\frac{a}{n}\int_0^{\infty} \frac{u}{(u+\theta)^{1-\sigma}}f_{U_n}(u)du\,,\label{eq.part 1 of NGGP mean}\\
& \frac{1}{n}\int_0^{\infty} \frac{u}{u+\theta}f_{U_n}(u)du\,.\label{eq.part 2 of NGGP mean}
\end{align}
We will find the limit of \eqref{eq.part 1 of NGGP mean} and then \eqref{eq.part 2 of NGGP mean}. For \eqref{eq.part 1 of NGGP mean} by the density of $U_n$, we have
\begin{align}
\frac{a}{n}\int_0^{\infty} \frac{u}{(u+\theta)^{1-\sigma}}f_{U_n}(u)du=\frac{\frac{a}{n} \int_0^{\infty} \frac{u^{n}}{(u+\theta)^{n+1-(n(\pi)+1)\sigma}}e^{-\frac{a}{\sigma}(u+\theta)^{\sigma}}du }{\int_0^{\infty} \frac{u^{n-1}}{(u+\theta)^{n-n(\pi)\sigma}}e^{-\frac{a}{\sigma}(u+\theta)^{\sigma}}du}\,.\label{eq2. part 1}
\end{align}
By the similar arguments in \cref{lemma1}, we use the Laplace method to find the limit of the nominator and denominator of \eqref{eq2. part 1}. Let 
\begin{align*}
&g_1(u)=\frac{u^{n}}{(u+\theta)^{n+1-(n(\pi)+1)\sigma}}e^{-\frac{a}{\sigma}(u+\theta)^{\sigma}}\,,
&g_2(u)=\frac{u^{n-1}}{(u+\theta)^{n-n(\pi)\sigma}}e^{-\frac{a}{\sigma}(u+\theta)^{\sigma}}\,.
\end{align*}
Thus,
\begin{align*}
&g_1'(u)=\lk n(u+\theta)+((n(\pi)+1)\sigma-(n+1))u-au(u+\theta)^{\sigma}\rk\frac{u^{n-1}}{(u+\theta)^{n+2-(n(\pi)+1)\sigma}}e^{-\frac{a}{\sigma}(u+\theta)^{\sigma}}\,,\\
&g_2'(u)=\lk (n-1)(u+\theta)+(n(\pi)\sigma-n)u-au(u+\theta)^{\sigma}\rk\frac{u^{n-2}}{(u+\theta)^{n+1-n(\pi)\sigma}}e^{-\frac{a}{\sigma}(u+\theta)^{\sigma}}\,
\end{align*}
As $n \rightarrow$, by the similar arguments in \cref{lemma1}, $g_1(u)$ and $g_2(u)$ attain their maximums  at $u_{1,n}$, $u_{2,n}$ that are both infinity large. Thus, $u_{1,n} \approx \lk \frac{(n(\pi)+1)\sigma -1}{a} \rk^{\frac{1}{\sigma}}-\theta$, and $u_{2,n} \approx \lk \frac{n(\pi)\sigma -1}{a} \rk^{\frac{1}{\sigma}}-\theta$. Therefore, followed by \eqref{eq2. part 1},
\begin{align}
&\lim_{n\rightarrow \infty} \frac{a}{n}\int_0^{\infty} \frac{u}{(u+\theta)^{1-\sigma}}f_{U_n}(u)du= \lim_{n \rightarrow \infty}\frac{a}{n}\frac{g_1(u_{1,n})}{g_2(u_{2,n})}\nonumber\\
&\qquad= \lim_{n \rightarrow \infty}\frac{a}{n} \frac{ \lk \lc\frac{(n(\pi)+1)\sigma -1}{a} \rc^{\frac{1}{\sigma}}-\theta\rk ^n  \lc\frac{(n(\pi)+1)\sigma -1}{a} \rc^{n(\pi) +1 -\frac{n+1}{\sigma}} e^{-(n(\pi)+1)-1/\sigma}}{\lk \lc\frac{n(\pi)\sigma -1}{a} \rc^{\frac{1}{\sigma}}-\theta\rk ^{n-1}  \lc\frac{n(\pi)\sigma -1}{a} \rc^{n(\pi)  -\frac{n}{\sigma}} e^{-n(\pi)-1/\sigma}}\nonumber\\
&\qquad = \lim_{n \rightarrow \infty} \frac{\lc(n(\pi)+1)\sigma -1 \rc^{(n(\pi)+1)-\frac{1}{\sigma}}}{n\lc n(\pi)\sigma -1 \rc^{n(\pi)-\frac{1}{\sigma}}}\frac{e^{-\theta a ^{1/\sigma} \lc\frac{n}{(n(\pi)+1)\sigma -1 }\rc ^{1/\sigma}n^{1-1/\sigma}}}{e^{-\theta a ^{1/\sigma} \lc\frac{n}{n(\pi)\sigma -1 }\rc ^{1/\sigma}(n-1)n^{-1/\sigma}}}e^{-1} \,.\label{eq3: part 1}
\end{align}
 Recall \cref{remark: ratio}, when $P_0$ is discrete, $\lim_{n \rightarrow \infty}\frac{n(\pi)}{n} = 0$,almost surely.  The limit in \eqref{eq3: part 1} becomes
 \begin{align*}
 &\lim_{n\rightarrow \infty} \frac{a}{n}\int_0^{\infty} \frac{u}{(u+\theta)^{1-\sigma}}f_{U_n}(u)du\\
 &\quad= \lim_{n\rightarrow \infty}  \frac{1}{n}\frac{\lc(n(\pi)+1)\sigma -1 \rc^{(n(\pi)+1)-\frac{1}{\sigma}}}{\lc n(\pi)\sigma -1 \rc^{n(\pi)-\frac{1}{\sigma}}}e^{-\theta a ^{1/\sigma} \lk \lc\frac{1}{(n(\pi)+1)\sigma -1 }\rc ^{1/\sigma}  n -\lc\frac{1}{n(\pi)\sigma -1 }\rc ^{1/\sigma}(n-1)\rk}e^{-1}\\
 &\quad = 0\,,
 \end{align*}
 where the  exponential part in the last equation converges to $0$ by the fact that $\lc\frac{1}{(n(\pi)+1)\sigma -1 }\rc ^{1/\sigma}-\lc\frac{1}{n(\pi)\sigma -1 }\rc ^{1/\sigma} >0$ for $\sigma \in (0,1)$.

 When $P_0$ is continuous,  $\lim_{n \rightarrow \infty}\frac{n(\pi)}{n} = 1$,almost surely. The limit in \eqref{eq3: part 1} becomes
  \begin{align*}
 &\lim_{n\rightarrow \infty} \frac{a}{n}\int_0^{\infty} \frac{u}{(u+\theta)^{1-\sigma}}f_{U_n}(u)du\\
 &\quad= \lim_{n\rightarrow \infty}  \frac{1}{n}\frac{\lc(n+1)\sigma -1 \rc^{(n+1)-\frac{1}{\sigma}}}{\lc n\sigma -1 \rc^{n-\frac{1}{\sigma}}}e^{-\theta a ^{1/\sigma} \lk \lc\frac{1}{(n)+1)\sigma -1 }\rc ^{1/\sigma}  n -\lc\frac{1}{n\sigma -1 }\rc ^{1/\sigma}(n-1)\rk}e^{-1}\\
 &\quad =\lim_{n\rightarrow \infty}\frac{\lc(n+1)\sigma -1 \rc}{n} \frac{\lc(n+1)\sigma -1 \rc^{-\frac{1}{\sigma}}}{\lc n\sigma -1 \rc^{-\frac{1}{\sigma}}} \lc 1+\frac{\frac{n\sigma}{n\sigma-1}}{n}\rc^ne^{-\theta a ^{1/\sigma} \lk \lc\frac{1}{(n)+1)\sigma -1 }\rc ^{1/\sigma}  n -\lc\frac{1}{n\sigma -1 }\rc ^{1/\sigma}(n-1)\rk}e^{-1}\\
&\quad = \sigma\,,
 \end{align*}
 where we emphasis that $\frac{1}{\sigma}>1$ when dealing with the convergence of the exponential part.
 
 By using the same arguments above for finding the limit of \eqref{eq.part 1 of NGGP mean}, we can find the limit of \eqref{eq.part 2 of NGGP mean}. We omit the details of the computation and can obtain the following results.
 
 When $P_0$ is discrete, 
 \begin{align*}
 &\lim_{n \rightarrow \infty} \frac{1}{n}\int_0^{\infty} \frac{u}{u+\theta}f_{U_n}(u)du =1\,.
 \end{align*}
 Thus,
 \begin{align*}
 &\lim_{n\rightarrow \infty}\E[P|\bX]= \lim_{n \rightarrow \infty} \frac{1}{n}\sum_{j=1}^{n(\pi)} (n_j-\sigma)\delta_{Y_j}=\lim_{n \rightarrow \infty}\mathbb{P}_n+ \frac{\sigma n(\pi)}{n}(H-\tilde{\mathbb{P}}_n)=P_0 \,,
 \end{align*}
where the last equation is due to $\lim_{n \rightarrow \infty}\frac{n(\pi)}{n} = 0$ and the Borel–Cantelli lemma . That is to say, the result in \eqref{eq: discrete case mean} is completed by combining the limit of \eqref{eq.part 1 of NGGP mean} and \eqref{eq.part 2 of NGGP mean} when $P_0$ is discrete.

 When $P_0$ is continuous, 
 \begin{align*}
 &\lim_{n \rightarrow \infty} \frac{1}{n}\int_0^{\infty} \frac{u}{u+\theta}f_{U_n}(u)du=1-\sigma\,. 
 \end{align*}
 Thus, combining the limit of \eqref{eq.part 1 of NGGP mean} and \eqref{eq.part 2 of NGGP mean}, we have
 \begin{align*}
 &\lim_{n\rightarrow \infty}\E[P|\bX]=  \lim_{n \rightarrow \infty}\sigma H+  \frac{1}{n}\sum_{j=1}^{n} (1-\sigma)\delta_{X_j}=\lim_{n\rightarrow \infty}\sigma H+ (1-\sigma)\mathbb{P}_n = \sigma H+ (1-\sigma)P_0 \,.
 \end{align*}
 Thus the proof of the result in \eqref{eq: continuous case mean} is completed.

\end{proof}

With the \cref{lemma: equivalent}, it is sufficient to proof \cref{thm: BVM} by only showing the convergences \eqref{eq1: BVM dicrete} and \eqref{eq1: BVM continuous}. The following lemma plays an important role in the proof of \cref{thm: BVM}. Here, we recall that an \textit{envelop function} of $\mathbb{F}$ is a measurable function $f_e \mathbb{X} \rightarrow \mathbb{R}$ such that $|f|<f_e$, for any $f \in \mathbb{F}$. 

\begin{lemma}\label{lemma: CLT of NGGP}
Let $\mathbb{F}$ be a finite set of $H-$square integrable functions. Assume that $n(\pi) \rightarrow \infty$ as $n\rightarrow \infty$, which includes the case when $P_0$ is continuous so that $n(\pi)=n$ and the case when $P_0$ is discrete but $n(\pi)$ converges to $\infty$ with a lower rate than $n$ do. Then 
\begin{align}
\sqrt{\sigma_n n(\pi)}(P_{U_n}-H)|\bX \leadsto \sqrt{1-\sigma}\mathbb{B}_{H}^o\,, \qquad \qquad a.s.,\label{eq. CLT lemma}
\end{align}
in $\mathbb{R}^{\mathbb{F}}$. The convergence holds a.s. in $l^{\infty}(\mathbb{F})$ with an envelop function $f_e$ such that $H(f_e^2)<\infty$, and thus the central limit theorem holds for $P_{U_n}|\bX$ in $l^{\infty}(\mathbb{F})$.
\end{lemma}
\begin{proof}
The proof relies on the stick-breaking representation of $P_{U_n}$ in \citep{favaro2016} and the functional central limit theorem of NGGP in \citep{huzhang2022functional}. And similarly as discussed in the proof of last lemma, we use $\sigma$ instead of $\sigma_n$ to make the interpretation easy to read.

By section 4.2 in \citep{favaro2016}, $P_{U_n}$ admits a stick-breaking representation with dependent stick-breaking weights $\{v_i\}_{i=1}^{\infty}$, and the joint distribution of $\{v_i\}_{i=1}^{\infty}$ are given \citep{huzhang2022functional} as 
\begin{align}
&f(v_1,\cdots,v_k)=\frac{\beta_n^k\sigma^{k-1}}{[\Gamma(1-\sigma)]^k\Gamma(k\sigma)}\prod_{i=1}^kv_i^{-\sigma}(1-v_i)^{-(k-i)\sigma-1}e^{-\frac{\beta_n}{\prod_{i=1}^k(1-v_i)^{\sigma}}}\nonumber\\
&\qquad\qquad\qquad\qquad  \times
\int_0^{\infty}(1-(1+t)^{-\frac{1}{\sigma}})^{k\sigma-1}(1+t)^{k-1}e^{-\frac{\beta_n t}{\prod_{i=1}^k(1-v_i)^{\sigma}}}dt\,,\label{eq1. CLT lemma}
\end{align}
where $\beta_n=\frac{a(u+\theta)^{\sigma}}{\sigma}$. We will follow the same idea as in the proof of Proposition 3.4 and the theorem 4.4 in \citep{huzhang2022functional}. To obtain the similar result as the Proposition 3.4  in \citep{huzhang2022functional}, we will consider the asymptotic result of the following quantity as $n \rightarrow \infty$. 
\begin{align}
&\E\lk \sum_{k=1}^{\infty} w_k^p |\bX\rk =\E \lk \E\lk \sum_{k=1}^{\infty} w_k^p | U_n =u, \bX\rk \rk=\E \lk \E\lk \sum_{k=1}^{\infty} v_k^p\prod_{l=1}^{k-1}(1-v_l)^p | U_n =u, \bX\rk \rk \nonumber\\
&=\int_0^{\infty} \frac{\beta_n^k\sigma^{k-1}}{[\Gamma(1-\sigma)]^k\Gamma(k\sigma)}\prod_{i=1}^kv_i^{-\sigma}(1-v_i)^{-(n-i)\sigma-1}e^{-\frac{\beta_n}{\prod_{i=1}^k(1-v_i)^{\sigma}}}\nonumber\\
&\qquad\qquad\qquad\qquad  \times
\int_0^{\infty}(1-(1+t)^{-\frac{1}{\sigma}})^{k\sigma-1}(1+t)^{k-1}e^{-\frac{\beta_n t}{\prod_{i=1}^k(1-v_i)^{\sigma}}}dt f_{U_n}(u) du\,,\label{eq2. CLT lemma}
\end{align}
where $p$ is any positive integer.  To evaluate \eqref{eq2. CLT lemma} as $n \rightarrow \infty$, we shall have a further analysis of the integral with respect to $u$, which is the only term that relates to $n$. Consider the following integral for any $b>0$, and any positive integer $k$.
\begin{align}
& \int_0^{\infty} \beta_n^k e^{b\beta_n}f_{U_n}(u) du=\frac{\lc \frac{a}{\sigma}\rc^k \int_0^{\infty} \frac{u^{n-1}}{(u+\theta)^{n-(n(\pi)+k)\sigma}}e^{-\frac{(b+1)a}{\sigma}(u+\theta)^{\sigma}}du }{\int_0^{\infty} \frac{u^{n-1}}{(u+\theta)^{n-n(\pi)\sigma}}e^{-\frac{a}{\sigma}(u+\theta)^{\sigma}}du}\nonumber\\
&= \frac{\lc \frac{a}{\sigma}\rc^k \lc \int_0^{M} \frac{u^{n-1}}{(u+\theta)^{n-(n(\pi)+k)\sigma}}e^{-\frac{(b+1)a}{\sigma}(u+\theta)^{\sigma}}du + \int_M^{\infty} \frac{u^{n-1}}{(u+\theta)^{n-(n(\pi)+k)\sigma}}e^{-\frac{(b+1)a}{\sigma}(u+\theta)^{\sigma}}du\rc}{\int_0^{\infty} \frac{u^{n-1}}{(u+\theta)^{n-n(\pi)\sigma}}e^{-\frac{a}{\sigma}(u+\theta)^{\sigma}}du} \,,\label{eq3. CLT lemma}
\end{align}
for any $M>0$. For any $n$ and any $M$, we have 
\begin{align*}
&\int_0^{M} \frac{u^{n-1}}{(u+\theta)^{n-(n(\pi)+k)\sigma}}e^{-\frac{(b+1)a}{\sigma}(u+\theta)^{\sigma}}du\\
&=\int_0^{M} \lc \frac{u}{(u+\theta)^{1-\sigma}} \rc^{n-1} \frac{1}{(u+\theta)^{(n-n(\pi)-k-1)\sigma+1}}e^{-\frac{(b+1)a}{\sigma}(u+\theta)^{\sigma}}du\\
&\leq \lc \frac{M}{(M+\theta)^{1-\sigma}} \rc^{n-1} \frac{(M+\theta)^{(n(\pi)+k+1-n)\sigma}}{(n(\pi)+k+1-n)\sigma} e^{-\frac{(b+1)a}{\sigma}(\theta)^{\sigma}}\,,
\end{align*}
which goes to $0$ as $n \rightarrow \infty$ due to the fact that either $\lim_{n \rightarrow \infty}\frac{n(\pi)}{n}=0$ or $n(\pi)=n$.  The last inequality holds because $\lc \frac{u}{(u+\theta)^{1-\sigma}} \rc^{n-1}$ is nondecreasing in $u$ for any $\sigma\in[0,1)$. Thus when $n \rightarrow \infty$, we have 
\begin{align*}
\int_0^{\infty} \beta_n^k e^{b\beta_n}f_{U_n}(u) du = \frac{\lc \frac{a}{\sigma}\rc^k \int_M^{\infty} \frac{u^{n-1}}{(u+\theta)^{n-(n(\pi)+k)\sigma}}e^{-\frac{(b+1)a}{\sigma}(u+\theta)^{\sigma}}du }{\int_0^{\infty} \frac{u^{n-1}}{(u+\theta)^{n-n(\pi)\sigma}}e^{-\frac{a}{\sigma}(u+\theta)^{\sigma}}du}\,.
\end{align*}
This would imply
\begin{align}
&\E\lk \sum_{k=1}^{\infty} w_k^p |\bX\rk =\int_{M_n}^{\infty} \frac{\beta_n^k\sigma^{k-1}}{[\Gamma(1-\sigma)]^k\Gamma(k\sigma)}\prod_{i=1}^kv_i^{-\sigma}(1-v_i)^{-(n-i)\sigma-1}e^{-\frac{\beta_n}{\prod_{i=1}^k(1-v_i)^{\sigma}}}\nonumber\\
&\qquad\qquad\qquad\qquad  \times
\int_0^{\infty}(1-(1+t)^{-\frac{1}{\sigma}})^{k\sigma-1}(1+t)^{k-1}e^{-\frac{\beta_n t}{\prod_{i=1}^k(1-v_i)^{\sigma}}}dt f_{U_n}(u) du\,,\label{eq4. CLT lemma}
\end{align}
in which we choose $M=M_n$ that goes to $\infty$ as $n \rightarrow \infty$. In this case, when $n \rightarrow \infty$, $\beta_n \rightarrow \infty$ as well and we are safe to use the results in Proposition 3.4  in \citep{huzhang2022functional} to obtain that when $n \rightarrow \infty$ (thus $n(\pi) \rightarrow \infty$)
\begin{align*}
\mathbb{E}\left[\sum_{n=1}^{\infty}w_n^2\right]= \int_{M_n}^{\infty} \lc \frac{1-\sigma}{a(u+\theta)^{\sigma}}\rc f_{U_n}(u)du=\frac{1-\sigma}{n(\pi)\sigma}+o\left(\frac{1}{n(\pi)}\right)\,,
\end{align*}
where the last equation can be computed by the same argument as in \eqref{eq2. part 1} and the computation afterwards.  The result of \eqref{eq. CLT lemma} follows immediately by applying the theorem 4.4 in \citep{huzhang2022functional}.
%That is to say, when $n \rightarrow \infty$ 
%\begin{align*}
%\frac{\frac{\sigma}{a} \int_{M_n}^{\infty} \frac{u^{n-1}}{(u+\theta)^{n-(n(\pi)-1)\sigma}}e^{-\frac{a}{\sigma}(u+\theta)^{\sigma}}du }{\int_0^{\infty} \frac{u^{n-1}}{(u+\theta)^{n-n(\pi)\sigma}}e^{-\frac{a}{\sigma}(u+\theta)^{\sigma}}du}
%\end{align*} 
%Furthermore, when $p=p_1=\cdots=p_k=2$ 
%   and when $\sigma=\frac{1}{m}$ for some arbitrarily fixed 
%    integer $m \geq 2$, we have  
%\begin{align}
%&\mathbb{E}\left[\sum_{n=1}^{\infty}w_n^2\right]=\frac{1}{a}+o\left(\frac{1}{a}\right) \quad\hbox{as $a\rightarrow \infty$},\label{6.5}\\
%&\mathbb{E}\left[\sum_{1 \leq i_1<i_2<\cdots <i_k<\infty}w_{i_1}^{2}w_{i_2}^{2}\cdots w_{i_k}^{2}\right]=\frac{1}{k!a^k}+o\left(\frac{1}{a^k}\right) \quad\hbox{as $a\rightarrow \infty$}.\label{6.6}
%\end{align} 
\end{proof}
By the above lemma and its proof, it is interesting to see that when $n \rightarrow \infty$, we can have $P_{n(\pi)}\overset{d}{=}P_{U_n}$, where $P_{n(\pi)}\sim \NGGP(n(\pi), \sigma,\theta,H)$ for any $n(\pi)\rightarrow \infty$. Thus, we can replace $P_{U_n}$ by $P_{n(\pi)}$ in the proof of \cref{thm: BVM}, the benefit of such replacement is $P_{n(\pi)}$ is independent of $\kappa_n$ when $n \rightarrow \infty$.

The next lemma provides the convergence of $\kappa_n$.
\begin{lemma}\label{lemma: convergence of kappa}
\begin{itemize}
\item[(i)] If $P_0$ is discrete,  when $n \rightarrow \infty$,
\begin{align}
\sqrt{n}(\kappa_n-\frac{\sigma_n n(\pi)}{n}) \leadsto 0\, \qquad\quad a.s.\nonumber
\end{align}
\item[(ii)] If $P_0$ is continuous, when $n \rightarrow \infty$, $\kappa_n \rightarrow \sigma$ in probability.
\end{itemize}
\end{lemma}
\begin{proof}
We shall compute the moments of $\kappa_n=\frac{T_{(U_n)}}{T_{(U_n)}+\sum_{j=1}^{n(\pi)}J_j}$ by the same method that we use in the proof of \cref{them. moment}. To make it clear, we present the details for $\E[\kappa_n]$ as follows.

\begin{align*}
&\E[\kappa_n]=\E \lk \E[\kappa_n|U_n]\rk=\E \lk \E \lk \int_0^{\infty} e^{-(T_{(U_n)}+\sum_{j=1}^{n(\pi)}J_j)y}T_{(U_n)}dy | U_n\rk\rk\\
&=\E\lk \int_0^{\infty} \lc-\frac{d}{dy}\E\lk e^{- T_{(U_n)} y }|U_n\rk \rc \prod_{j=1}^{n(\pi)} \E\lk e^{-yJ_j}|U_n\rk dy \rk\\
&=\E \lk \int_0^{\infty} a (y+U_n+\theta)^{\sigma_n-1} e^{-\frac{a}{\sigma_n} \lc (y+U_n+\theta)^{\sigma_n}-(U_n+\theta)^{\sigma_n} \rc} \lc\frac{U_n+\theta}{y+U_n+\theta} \rc^{n-n(\pi)\sigma_n} dy \rk\,,
\end{align*}
where the last equation is a direct use of the Laplace transform of $T_{U_n}$ and the distribution of $J_j$ in \cref{lemma: NGGP posterior}. Solving the expectation with respect to $U_n$, we have 
\begin{align*}
\E[\kappa_n]=\frac{\int_0^{\infty}\int_0^{\infty}a \frac{u^{n-1}}{(y+u+\theta)^{(n+1)-(n(\pi)+1)\sigma_n}}e^{-\frac{a}{\sigma_n}(y+u+\theta)^{\sigma_n}}dy du}{\int_0^{\infty} \frac{u^{n-1}}{(u+\theta)^{n-n(\pi)\sigma_n}}e^{-\frac{a}{\sigma_n}(u+\theta)^{\sigma_n}}du}\,.
\end{align*}
By the substitution $v=u+y$ and $z=u$ on the nominator of the above form.
\begin{align*}
\E[\kappa_n]=\frac{a}{n}\frac{\int_0^{\infty} \frac{v^{n}}{(v+\theta)^{(n+1)-(n(\pi)+1)\sigma_n}}e^{-\frac{a}{\sigma_n}(v+\theta)^{\sigma_n}}dv}{\int_0^{\infty} \frac{u^{n-1}}{(u+\theta)^{n-n(\pi)\sigma_n}}e^{-\frac{a}{\sigma_n}(u+\theta)^{\sigma_n}}du}\,,
\end{align*}
which implies $\E[\kappa_n]=\frac{(n(\pi)+1)\sigma_n-1}{n}$ by the analysis of \ref{eq.part 1 of NGGP mean}. And we have $\E[\kappa_n]\rightarrow 0$ if $P_0$ is discrete and $\E[\kappa_n]\rightarrow \sigma$ if $P_0$ is continuous when $n\rightarrow \infty$.

Similarly, we can obtain the second moment of $\kappa_n$ in the same way as $n \rightarrow \infty$.
\begin{align*}
\E[\kappa_n^2]=\frac{\lk (n(\pi)+2)\sigma_n-1\rk^2}{n(n+1)}+\frac{\lk (n(\pi)+2)\sigma_n-1\rk(1-\sigma_n)}{an(n+1)}\,,
\end{align*}
followed by which, we have $$\Var[\kappa_n]=\frac{\lk (n(\pi)+2)\sigma_n-1\rk(1-\sigma_n)}{an(n+1)}+\frac{2n(\pi)+3\sigma_n^2-2\sigma_n}{n(n+1)}-\frac{\lk(n(\pi)+1)\sigma_n-1 \rk^2}{n^2(n+1)}.$$  And $\lim_{n \rightarrow \infty} \Var[\kappa_n]=0$ for both continuous and discrete $P_0$. 
%In particular, if $P_0$ is discrete, $\Var[\sqrt{n}\kappa_n] \rightarrow 0$ as well when $n \rightarrow \infty$.
This complete the proof of the lemma.
\end{proof}
Now, \cref{thm: BVM} can be proved by using the previous lemmas. And we give the details as follows.
\begin{proof}{\bf Proof of \cref{thm: BVM}}

We proof \cref{thm: BVM} in two parts corresponding to when $P_0$ is discrete and when $P_0$ is continuous. We denote $R_n=\sum_{j=1}^{n(\pi)} \frac{J_j}{\sum_{j=1}^{n(\pi)}} \delta_{Y_j}$. Then, $R_n=\sum_{j=1}^{n(\pi)} D_{n,j}\delta_{Y_j}$.

\textbf{(i) When $P_0$ is discrete.}

It is convenient to decompose $\sqrt{n}\left(P-\mathbb{P}_n-\frac{\sigma_n n(\pi)}{n}(H-\tilde{\mathbb{P}}_n) \right)|\bX $ as 
\begin{align}
&\sqrt{n}\lc \kappa_n-\frac{\sigma_n n(\pi)}{n} \rc \lc P_{U_n}-R_n \rc +  \lc \sqrt{\sigma_n  n(\pi)} (P_{U_n}-H) \sqrt{ \frac{\sigma_n n(\pi)}{n}}\rc \nonumber\\
&\qquad\qquad +\sqrt{n}\lc R_n\lc 1-\frac{\sigma_n n(\pi)}{n} \rc -\lc \mathbb{P}_n-\frac{\sigma_n n(\pi)}{n} \tilde{\mathbb{P}}_n \rc\rc |\bX \,. \label{eq1. proof of BVM}
\end{align}

The first term in decomposition \eqref{eq1. proof of BVM} converges to $0$ by using \cref{lemma: convergence of kappa} and the fact that $P_{U_n}-R_n$ is uniformly bounded. The second term in decomposition \eqref{eq1. proof of BVM} converges to $0$ by using \cref{lemma: CLT of NGGP} and the fact that $\frac{\sigma_n n(\pi)}{n} \rightarrow 0$ a.s.. And the convergence for the first two terms in decomposition \eqref{eq1. proof of BVM} holds for both $n(\pi)$ is finite and goes to $\infty $ when $n \rightarrow \infty$.

The convergence of the last term in decomposition \eqref{eq1. proof of BVM} relaying on the gamma representation of $D_{n,j}$ in $R_n$. For each $j \in \{1, \cdots, n(\pi)\}$, we rewrite
\begin{align*}
D_{n,j}=\frac{\gamma_{j,0}+\sum_{i=1}^{n_j-1} \gamma_{j,i}}{\sum_{j=1}^{n(\pi)} \lc \gamma_{j,0}+\sum_{i=1}^{n_j-1} \gamma_{j,i} \rc}\,,
\end{align*}
where the independent random variables $\gamma_{j,0} \sim G(1-\sigma_n, 1)$ and $\gamma_{j,i} \sim G(1,1)$ for all $j$ and all $i$. That is to say, there are $n$ $\gamma_{j,i}$'s ($i$ can take $0$) for $j \in \{1, \cdots, n(\pi)\}$, during which, there are $n(\pi)$ independent $G(1-\sigma_n, 1)$ random variables and $n-n(\pi)$ independent $G(1, 1)$ random variables. Relabel all these $n$ gamma random variables as $\{\tilde{\gamma}_{n,l}\}_{l=1}^n$ (the order doesn't matter). Then
\begin{align}
R_n=\sum_{j=1}^{n(\pi)} D_{n,j}\delta_{Y_j}=\frac{n^{-1} \sum_{l=1}^n \tilde{\gamma}_{n,l} \delta_{X_l} }{n^{-1} \sum_{l=1}^n \tilde{\gamma}_{n,l} }\,.
\end{align} 
To make the interpretation clear, we denote $R_nf =\frac{ \bar{R}_nf}{\bar{R}_n1}$, where $\bar{R}_nf=\frac{\sum_{l=1}^n \tilde{\gamma}_{n,l} f(X_l)}{n}$ and $\bar{R}_n1=\frac{\sum_{l=1}^n \tilde{\gamma}_{n,l}}{n}$. Thus,
\begin{align}
&\sqrt{n}\lc R_n\lc 1-\frac{\sigma_n n(\pi)}{n} \rc -\lc \mathbb{P}_n-\frac{\sigma_n n(\pi)}{n} \tilde{\mathbb{P}}_n \rc\rc f\nonumber\\
&=-R_nf \sqrt{n} \lc \bar{R}_n1-\lc 1-\frac{\sigma_n n(\pi)}{n}\rc \rc+\sqrt{n} \lc \bar{R}_nf-\lc \mathbb{P}_nf-\frac{\sigma_n n(\pi)}{n} \tilde{\mathbb{P}}_nf \rc \rc\,.\label{eq2. proof of BVM}
\end{align}
It is clear that $\mathbb{P}_nf+\frac{\sigma_n n(\pi)}{n} \tilde{\mathbb{P}}_nf \rightarrow P_0f$  almost surely, by the Borel-Cantelli lemma and the fact that $\mathbb{F}$ is a finite set such that $P_0(f^2) < \infty$. Thus, $\mathbb{F}$ is a $P_0-$Donsker class. By the distributions and independence of $\{\tilde{\gamma}_{n,l}\}_{l=1}^n$, we have 
\begin{align*}
&\E\lk\bar{R}_n1 \rk=\E\lk \frac{\sum_{l=1}^n \tilde{\gamma}_{n,l}}{n} \rk=1-\frac{\sigma_n n(\pi)}{n} \rightarrow 1 \qquad \text{a.s.}\,,\\
&\Var\lk \bar{R}_n1\rk=\Var\lk \frac{\sum_{l=1}^n \tilde{\gamma}_{n,l}}{n} \rk=\frac{1}{n}\sum_{l=1}^n\Var \lk \tilde{\gamma}_{n,l}\rk=\frac{1}{n}-\frac{\sigma_n n(\pi)}{n}\rightarrow 0 \qquad \text{a.s.}\,.
\end{align*}
Thus, we have the convergence $\sqrt{n} \lc \bar{R}_n1-\lc 1-\frac{\sigma_n n(\pi)}{n}\rc \rc \leadsto  0$. By noting that $R_n$ is uniformly bounded, we obtain $-R_nf \sqrt{n} \lc \bar{R}_n1-\lc 1-\frac{\sigma_n n(\pi)}{n}\rc \rc \leadsto 0$. To find the convergence of $\sqrt{n} \lc \bar{R}_nf-\lc \mathbb{P}_nf-\frac{\sigma_n n(\pi)}{n} \tilde{\mathbb{P}}_nf \rc \rc$, we follow the similar way and check the Linderberg-Feller condition as follows.
\begin{align*}
&\E\lk\bar{R}_nf \rk=\frac{1}{n}\sum_{l=1}^n\E\lk  \tilde{\gamma}_{n,l} \rk f(X_l)=\mathbb{P}_nf-\frac{\sigma_n n(\pi)}{n} \tilde{\mathbb{P}}_nf \,,\\
&\Var\lk\bar{R}_nf \rk=\frac{1}{n^2}\sum_{l=1}^n\Var\lk  \tilde{\gamma}_{n,l} \rk f^2(X_l)=\frac{1}{n}\lc \mathbb{P}_nf^2-\frac{\sigma_n n(\pi)}{n} \tilde{\mathbb{P}}_nf^2 \rc\,,\\
&\frac{1}{n	}\sum_{l=1}^n \E \lk \tilde{\gamma}_{n,l}^2f^2(X_l) \mathbbm{1}_{|\tilde{\gamma}_{n,l}f(X_l)|>\epsilon \sqrt{n}}  \rk \\
&\quad\leq \max \lc \E \lk \gamma_{j,0}^2f^2(X_l) \mathbbm{1}_{|\gamma_{j,0}|\max_{1\leq l \leq n}| f(X_l)|>\epsilon \sqrt{n}}  \rk, \E \lk \gamma_{j,i}^2f^2(X_l) \mathbbm{1}_{|\gamma_{j,i}|\max_{1\leq l \leq n}| f(X_l)|>\epsilon \sqrt{n}}  \rk\rc \mathbb{P}_nf^2\,,
\end{align*}
where the last inequality is a verification of Linderberg-Feller condition, and the right hand side converges to $0$ for every sequence $\bX$, since $P_0f^2<\infty$ and $\max_{1\leq l \leq n}| f(X_l)/\sqrt{n} \rightarrow 0$. By the Cram\'{e}r-Wold device and the linearity of $f$, we have $\sqrt{n} \lc \bar{R}_nf-\lc \mathbb{P}_nf-\frac{\sigma_n n(\pi)}{n} \tilde{\mathbb{P}}_nf \rc \rc \leadsto \mathbb{B}_{P_0}^of$ for any $f \in \mathbb{F}$.

To show the convergence in $l^{\infty}(\mathbb{F})$ for any $P_0-$Donsker class, we shall prove the asymptotic tightness, see e.g., Theorem 1.5.4 in \citep{van1996}. The multipliers of the multiplier process $\frac{1}{\sqrt{n}}\sum_{l=1}^n(\tilde{\gamma}_{n,l}-\E \lk \tilde{\gamma}_{n,l}\rk) f(X_l)$ are independent with $0$ means. Thus, the multiplier central limit theorem in Theorem 2.9.7 \citep{van1996} can be applied once we have the following inequality for any collection $\mathbb{H}$ of functions.
\begin{align*}
\E_{\tilde{\gamma}}\left\| \sum_{l=1}^n  (\tilde{\gamma}_{n,l}-\E \lk \tilde{\gamma}_{n,l}\rk) f(X_l) \right\|_{\mathbb{H}}^* \leq \E_{\tilde{\gamma}, \tilde{\gamma}'}\left\| \sum_{l=1}^n  (\tilde{\gamma}_{n,l}-\E \lk \tilde{\gamma}_{n,l}\rk + \tilde{\gamma}_{n,l}'-\E \lk \tilde{\gamma}_{n,l}'\rk) f(X_l) \right\|_{\mathbb{H}}^*\,,
\end{align*}
by Jensen's inequality, for any random variable $\tilde{\gamma}_{n,l}'$ independent of $\tilde{\gamma}_{n,l}$. It is safe to choose all $\tilde{\gamma}_{n,l}' \overset{iid}{\sim} G(1,1)$ and $\tilde{\gamma}_{n,l}\overset{iid}{\sim} G(1,1)$. Then, the multiplier central limit theorem that is given as Theorem 2.9.7 (see also 2.9.6,2.9.9,3.6.13) in \citep{van1996}, the asymptotic tightness follows immediately. (We apply the inequality with $\mathbb{H}$ to be the set of $f_1-f_2$ for any $f_1, f_2 \in \mathbb{F}$, with $L_2(P_0)$ norm of $f_1-P_0f_1-(f_2-P_0f_2)$ smaller than $\delta$.)

This complete the proof of the theorem when $P_0$ is discrete.

\textbf{(ii) When $P_0$ is continuous.}

In this case, $n(\pi)=n$. We can decompose $\sqrt{n}\left(P- \lk (1-\sigma)\mathbb{P}_n+\sigma H \rk \right)|\bX$ as 
\begin{align}
\sqrt{n} \lc P_{U_n}-H \rc \kappa_n+\sqrt{n} (1- \kappa_n) (R_n-\mathbb{P}_n)+\sqrt{n}(\kappa_n-\sigma_n)(H-\mathbb{P}_n)\,.\label{eq3. proof of BVM}
\end{align}

For the convergence of the first term in \eqref{eq3. proof of BVM}, we first use the discussion below the proof of \cref{lemma: CLT of NGGP} to use $P_n \overset{d}{=} P_{U_n}$ when $n\rightarrow \infty$, where $P_n \sim \NGGP(n, \sigma, \theta, H)$. Thus, we can consider the convergence of $\sqrt{n} \lc P_{n}-H \rc \kappa_n$ instead of $\sqrt{n} \lc P_{U_n}-H \rc \kappa_n$, the benefit of the former form is $P_n$ and $\kappa_n$ are independent. Thus, by \cref{lemma: convergence of kappa}, $\kappa_n \rightarrow \sigma$ in probability. By using the result in \cref{lemma: CLT of NGGP}, we have $\sqrt{n} \lc P_{n}-H \rc \leadsto \sqrt{\frac{1-\sigma}{\sigma}}\mathbb{B}_{H}^o$ a.s.. Thus, we have $\sqrt{n} \lc P_{U_n}-H \rc \kappa_n \leadsto \sqrt{\sigma(1-\sigma)}$.

For the second term in \eqref{eq3. proof of BVM}, $R_n=\sum_{j=1}^n D_{n,j} \delta_{X_j}$, where $D_{n,j}=\frac{\gamma_j}{\sum_{j=1}^n \gamma_j}$ with $\gamma \overset{iid}{\sim} G(1-\sigma_n, 1)$. A direct application of the result of Theorem 2.1 in \citep{praestgaard1993exchangeably} implies $\sqrt{n}(R_n-\mathbb{P}_n) \leadsto \frac{1}{\sqrt{1-\sigma}} \mathbb{B}_{P_0}^o $ a.s., in $l^{\infty}(\mathbb{F})$ if there is a $P_0-$square-integrable envelope function for $\mathbb{F}$. Furthermore, the convergence in probability is a direct application of Theorem 2.9.7  in \citep{van1996}. By noting the fact that $(1- \kappa_n) \rightarrow 1-\sigma$ in probability, we have $\sqrt{n} (1- \kappa_n) (R_n-\mathbb{P}_n) \leadsto \sqrt{1-\sigma}\mathbb{B}_{P_0}^o$ a.s.  in $l^{\infty}(\mathbb{F})$.

For the last term in \eqref{eq3. proof of BVM}, we follow the same argument as that in \cref{lemma: convergence of kappa} and will have $\Var[\sqrt{n}\kappa_n]=(1-\sigma_n)\sigma_n \rightarrow (1-\sigma)\sigma$, thus 
$\sqrt{n}(\kappa_n-\sigma_n) \leadsto \sqrt{\sigma(1-\sigma)}Z$. Furthermore, by the Borel-Cantelli lemma, $\mathbb{P}_n \rightarrow P_0$ a.s., and thus $\sqrt{n}(\kappa_n-\sigma_n)(H-\mathbb{P}_n) \leadsto \sqrt{\sigma(1-\sigma)}Z(H-P_0)$.

The result in Theorem \ref{thm: BVM} when $P_0$ is continuous follows by combining the convergences of the three terms in \eqref{eq3. proof of BVM}.
\end{proof}
\end{document}